\documentclass[12pt]{amsart}
\usepackage{amsmath}
\usepackage{latexsym}
\usepackage{amssymb}
\usepackage{graphicx}

\usepackage{physics}
\usepackage{tikz}    
\usepackage{mathdots}   
\usepackage{yhmath} 
\usepackage{cancel}  
\usepackage{color} 
\usepackage{siunitx} 
\usepackage{array} 
\usepackage{multirow} 
\usepackage{gensymb}
\usepackage{tabularx} 
\usepackage{booktabs} 
\usetikzlibrary{fadings}
\usetikzlibrary{patterns}
\usetikzlibrary{shadows.blur}
\usetikzlibrary{shapes}

\DeclareMathOperator{\co}{co}

\newtheorem{theorem}{Theorem}
\newtheorem{proposition}{Proposition}
\newtheorem{lemma}{Lemma}

\newtheorem{remark}{Remark}

\begin{document}
%
\def\R {{\mathbb{R}}}
\def\N {{\mathbb{N}}}
\def\C {{\mathbb{C}}}
\def\Z {{\mathbb{Z}}}
\def\phi{\varphi}
\def\epsilon{\varepsilon}
\def\ma{{\mathcal A}}
%
\def\tb#1{\|\kern -1.2pt | #1 \|\kern -1.2pt |} 
\def\Qed{\qed\par\medskip\noindent}
%

\title[The distance function]{The distance function \\
in the presence of an obstacle}  
\author{Paolo Albano} 
\address{Dipartimento di Matematica, Universit\`a di Bologna, Piazza di Porta San Donato 5, 40127 Bologna, Italy} 
\email{paolo.albano@unibo.it}
 \author{Vincenzo Basco}
\address{Thales Alenia Space, via Saccomuro, 24, Roma, Italy}
\email{vincenzobasco@thalesaleniaspace.com, vincenzobasco@gmail.com}
\author{Piermarco Cannarsa} 
\address{Dipartimento di Matematica, Universit\`a di Roma "Tor Vergata",  Roma, Italy} 
\email{cannarsa@mat.uniroma2.it}

\date{\today}

\begin{abstract}
We study the Riemannian distance function from a fixed point (a point-wise target) of Euclidean space in the presence of a compact obstacle bounded by a smooth hypersurface. First, we show that such a function is locally semiconcave  with a fractional modulus of order one half and that, near the obstacle,  this regularity is optimal. Then, in the Euclidean setting, we prove that the singularities of the distance function propagate, in the sense that each singular point belongs to a nontrivial singular continuum. 
Finally, we investigate  the lack of differentiability of the distance function when a convex obstacle is present.
\end{abstract}

\subjclass[2010]{49J52, 26A27, 26B25, 49L2}
\keywords{distance function, state contraints, semiconcave functions, singularities}

\maketitle

\section{Introduction and statement of the results}
\setcounter{equation}{0}
\setcounter{theorem}{0}
\setcounter{proposition}{0}  
\setcounter{lemma}{0}
\setcounter{corollary}{0} 
\setcounter{definition}{0}

We study the distance function $d$ from a fixed point of $\R^n$ in the presence of an obstacle, $\mathcal{O}$, bounded by a  $C^2$ hypersurface. 
We are interested in two main issues, namely:
\begin{itemize}
\item the regularity of $d$ up to the boundary of $\mathcal{O}$, and
\item    the analysis of the singularities of $d$
 (i.e. points where the distance function is not differentiable) on $\overline{\R^n\setminus \mathcal{O}}$.
\end{itemize}
For simplicity, we may assume that 
 \vspace{0,5cm}

 \noindent
(O) $\mathcal{O}$ is the closure of a bounded connected open subset of $\R^n$ with boundary of class $C^2$, such that $\R^n\setminus \mathcal{O}$ is connected. 

\vspace{0,5cm} 

In order to define the distance function, let us consider a family of positive definite quadratic forms with $C^2$ coefficients 
$$
\R^n\ni x\mapsto A(x) ,
$$
and the set of all the ``subunit'' curves 
$$
\Gamma =\{ \gamma \in AC(0,+\infty ; \overline{\R^n\setminus \mathcal{O}} )\ | \   \text{for a.e. }t\geq 0,\,  \langle A(\gamma (t)) \dot{\gamma}(t), \dot{\gamma}(t)\rangle \le 1\}. 
$$
For every $x\in  \overline{\R^n\setminus \mathcal{O}}$, we define 
$$
\Gamma [x]=\{ \gamma\in \Gamma \ | \  \gamma (0)=x \}.
$$
Let $k_0\in \R^n\setminus \mathcal{O}$ be fixed,  set 
$$
\tau (\gamma )=\inf \{ t\geq 0\ |\ \gamma (t)=k_0\}\in [0,+\infty ],
$$
and  consider the following constrained minimization problem
\begin{equation}\label{min}
\inf_{\gamma \in \Gamma [x]} \tau (\gamma ) .
\end{equation}
We observe that the problem above can be seen as a  constrained minimum time problem with a point-wise target.  
Then, the distance function of $x$ from $k_0$ is given by  
\begin{equation}\label{d}
d(x)=\inf_{\gamma \in \Gamma [x]}\tau  (\gamma),\qquad x\in \overline{\R^n\setminus\mathcal{O}}.
\end{equation}
\begin{remark}\label{r:1}
We observe that the assumption ``$\R^n\setminus \mathcal{O}$ is connected'' ensures that  $\Gamma [x]\not=\emptyset$, 
 for every $x\in \overline{\R^n\setminus\mathcal{O}}$, i.e.,
$d(x)$ is finite for every $x\in \overline{\R^n\setminus\mathcal{O}}$.
\end{remark} 
We recall that $d$ is the viscosity solution of a suitable boundary value problem for  the eikonal equation (see, e.g., Theorem X.1  in \cite{CDL}) 
\begin{equation}\label{eq:iconale}
\langle A^{-1}(x) Dd(x),Dd(x)\rangle =1\qquad \text{ in } \R^n\setminus ( \mathcal{O}\cup \{k_0\}),  
\end{equation} 
where $A^{-1}$ denotes the inverse matrix of $A$. 

Since $d$ is the value function of a minimum time problem with state constraints, $d$ is expected to be ``more'' than Lipschitz continuous but ``less'' than  differentiable.  The appropriate regularity class for $d$ is the one of semiconcave functions with fractional modulus. 

Given a set $U\subset \R^n$, we say that $u:U\longrightarrow \R$ is a fractionally semiconcave function on $U$ of exponent  $\alpha \in ]0,1]$   if $u$ is locally Lipschitz continuous on $U$ and   
there exists  $C>0$ such that 
\begin{equation}\label{eq:sca}
\lambda u(x)+(1-\lambda) u(y)-u(\lambda x + (1-\lambda )y)\le C\lambda (1-\lambda) |x-y|^{1+\alpha},
\end{equation}
for any $x,y\in U$   such that the line segment $[x,y]$ is contained in $U$ and for every $\lambda\in [0,1]$. In the case $\alpha =1$ we say that $u$ is semiconcave with linear modulus or linearly semiconcave. Furthermore, we call the constant $C$ given in \eqref{eq:sca} a semiconcavity constant for $u$ in $U$. We call $C$ a linear semiconcavity constant for $u$ when $\alpha =1$. 
We denote by $SC^\alpha (U)$ the set of all the fractionally semiconcave functions of exponent $\alpha$ in $U$. Such a property can obviously be made local, in which case we refer to the space $SC_{loc}^\alpha (U)$.
\begin{remark}
We recall that, if $U$ is an open set with $\partial U\not=\emptyset$ and $u$ satisfies \eqref{eq:sca}  on every compact set $K\subset\overline U$ for some constant $C_K$, then $u$ is locally Lipschitz continuous on $U$ (see e.g. \cite{CS})  but not on $\partial U$, in general. For this reason, in the above definition, we require $u$ to be 
locally Lipschitz continuous on $U$. 
\end{remark} 
For any function $u\in SC_{loc}^\alpha (U)$ and any fixed $x\in U$, we denote by $D^*u(x)$ the nonempty compact set of {\em reachable gradients} of $u$ at $x$, i.e.,
\begin{multline}\label{eq:dstar}
D^*u(x)=\{ p\in \R^n\ | \  \exists x_h\in \mbox{int}(U), 
x_h\to x, \  \exists Du(x_h)\to p\} ,
\end{multline}
where $\mbox{int}(U)$ stands for the interior of $U$.

We recall the following interior regularity property which is a direct consequence of the results in \cite{L} and \cite{A3}. 
\begin{theorem}[Interior regularity] \label{t:i}
Let $u$ be a viscosity solution of Equation \eqref{eq:iconale} and assume $x\mapsto A(x)$ to be a map of class $C^2$ taking values in the set of all positive definite $n\times n$ matrices. Then, we have that $u\in SC_{loc}^1( \R^n\setminus (\mathcal{O}\cup \{ k_0\} ))$.  
\end{theorem}
Our first goal is to study the regularity of $d$ up to the boundary of the obstacle $\mathcal{O}$. We have the following 
\begin{theorem}[Boundary Regularity]\label{t:rb}
 Under Assumption (O), let $k_0\in \R^n\setminus \mathcal{O}$, assume that $x\mapsto A(x)$ is a map of class $C^2$ taking values in the set of all positive definite matrices, and let $d$ be given by \eqref{d}. Then  $d\in SC^{\frac 12}_{loc}(\overline{ \R^n \setminus \mathcal{O}}\setminus  \{ k_0\})$.  
\end{theorem} 
The conclusion of Theorem \ref{t:rb} can be obtained by reducing our problem to a minimum energy problem and then applying the semiconcavity result of \cite{CCMW} for Tonelli type Hamiltonians. We give the reasoning in the appendix of this paper. 
\begin{remark}\label{re:connectedness}
Observe that, since the above conclusion is of local nature, Theorem~\ref{t:rb} holds true without assuming $\mathcal O$ to consist of a single connected component. 
\end{remark}
One may wonder if the regularity result above is somehow optimal. For this reason, let us consider the special case where $A(x)\equiv I$ (the $n\times n$ identity matrix) and the obstacle is the  ball $\overline{B_1(0)}$. Let us  split $\overline{\R^n\setminus \mathcal{O}}$ into two sets 
\begin{equation}\label{ik0}
I(k_0)=\{ x\in \overline{\R^n\setminus \mathcal{O}}\ |\ d(x)=|x-k_0|\}
\end{equation}
and 
\begin{equation}\label{sk0}
S(k_0) =\{  x\in \overline{\R^n\setminus \mathcal{O}}\ |\ |x-k_0|<d(x)\}= \overline{\R^n\setminus \mathcal{O}}\setminus I(k_0). 
\end{equation}
%


 
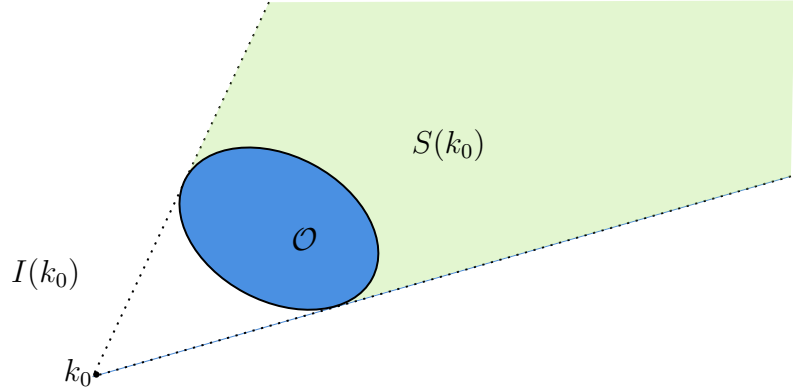
\tikzset{
pattern size/.store in=\mcSize, 
pattern size = 5pt,
pattern thickness/.store in=\mcThickness, 
pattern thickness = 0.3pt,
pattern radius/.store in=\mcRadius, 
pattern radius = 1pt}
\makeatletter
\pgfutil@ifundefined{pgf@pattern@name@_d8o746zxz}{
\pgfdeclarepatternformonly[\mcThickness,\mcSize]{_d8o746zxz}
{\pgfqpoint{0pt}{0pt}}
{\pgfpoint{\mcSize+\mcThickness}{\mcSize+\mcThickness}}
{\pgfpoint{\mcSize}{\mcSize}}
{
\pgfsetcolor{\tikz@pattern@color}
\pgfsetlinewidth{\mcThickness}
\pgfpathmoveto{\pgfqpoint{0pt}{0pt}}
\pgfpathlineto{\pgfpoint{\mcSize+\mcThickness}{\mcSize+\mcThickness}}
\pgfusepath{stroke}
}}
\makeatother

 
\tikzset{
pattern size/.store in=\mcSize, 
pattern size = 5pt,
pattern thickness/.store in=\mcThickness, 
pattern thickness = 0.3pt,
pattern radius/.store in=\mcRadius, 
pattern radius = 1pt}
\makeatletter
\pgfutil@ifundefined{pgf@pattern@name@_fi5ccg1f8}{
\pgfdeclarepatternformonly[\mcThickness,\mcSize]{_fi5ccg1f8}
{\pgfqpoint{0pt}{0pt}}
{\pgfpoint{\mcSize+\mcThickness}{\mcSize+\mcThickness}}
{\pgfpoint{\mcSize}{\mcSize}}
{
\pgfsetcolor{\tikz@pattern@color}
\pgfsetlinewidth{\mcThickness}
\pgfpathmoveto{\pgfqpoint{0pt}{0pt}}
\pgfpathlineto{\pgfpoint{\mcSize+\mcThickness}{\mcSize+\mcThickness}}
\pgfusepath{stroke}
}}
\makeatother
\tikzset{every picture/.style={line width=0.75pt}} 

\begin{tikzpicture}[x=0.75pt,y=0.75pt,yscale=-0.7,xscale=0.7]

\draw  [dash pattern={on 0.84pt off 2.51pt}]  (323,1) -- (196,271) ;
\draw [fill={rgb, 255:red, 78; green, 141; blue, 215 }  ,fill opacity=1 ] [dash pattern={on 0.84pt off 2.51pt}]  (699,126.5) -- (196,271) ;
\draw  [draw opacity=0][fill={rgb, 255:red, 184; green, 233; blue, 134 }  ,fill opacity=0.39 ] (701.5,-0.25) -- (699,126.5) -- (375,219) -- (262.87,128.37) -- (323,1) -- cycle ;
\draw  [draw opacity=0][pattern=_d8o746zxz,pattern size=45pt,pattern thickness=0.75pt,pattern radius=0pt, pattern color={rgb, 255:red, 0; green, 0; blue, 0}] (0,1.5) -- (85.91,1.37) -- (323,1) -- (196,271) -- (279,309.5) -- (1,308.5) -- cycle ;
\draw  [draw opacity=0][pattern=_fi5ccg1f8,pattern size=45pt,pattern thickness=0.75pt,pattern radius=0pt, pattern color={rgb, 255:red, 0; green, 0; blue, 0}] (262.87,128.37) -- (375,219) -- (699,126.5) -- (698,308.5) -- (267,309.5) -- (196,271) -- cycle ;
\draw  [fill={rgb, 255:red, 74; green, 144; blue, 226 }  ,fill opacity=1 ] (262.87,128.37) .. controls (276.49,102.87) and (317.63,98.27) .. (354.76,118.1) .. controls (391.9,137.93) and (410.96,174.67) .. (397.35,200.17) .. controls (383.73,225.67) and (342.59,230.27) .. (305.45,210.45) .. controls (268.32,190.62) and (249.25,153.87) .. (262.87,128.37) -- cycle ;

\draw (423.67,90.67) node [anchor=north west][inner sep=0.75pt]    {$S( k_{0})$};
\draw (173,256.33) node [anchor=north west][inner sep=0.75pt]    {$k_{0}$};

\draw (135,183.33) node [anchor=north west][inner sep=0.75pt]    {$I( k_{0})$};
\draw (337,162) node [anchor=north west][inner sep=0.75pt]    {$\mathcal{O}$};
\draw (190,265) node [anchor=north west][inner sep=0.75pt]  [font=\Huge]  {${\displaystyle .}$};

\end{tikzpicture}
We observe that  $I(k_0)\not=\emptyset$ is a closed set   and   $S(k_0)\not=\emptyset$ is a relatively open subset of $\overline{\R^n\setminus\mathcal{O}}$. Furthermore, we have that 
\begin{equation}\label{eq:dini}
d(x)=|x-k_0|,\qquad \text{ for every }x\in I(k_0).
\end{equation}
We begin with observing that, in general, linear semiconcavity does not hold up to the boundary. More precisely, we have the following 
\begin{proposition}\label{cm}
Let $\mathcal{O}=\overline{B_1(0)}$, let $k_0\in \R^n\setminus \mathcal{O}$, let $A(\cdot )\equiv I$, and let $d$ be given by \eqref{d}. 
Then, 

\begin{itemize}
\item[$(i)$] $d\notin SC^1(\overline{ B_\delta (x)\setminus \mathcal{O}})$, for every $x\in S(k_0)\cap \partial\mathcal{O}$ and $\delta >0$.    
\item[$(ii)$] $d\in SC^\alpha_{loc}(\overline{\R^n\setminus \mathcal{O}}\setminus \{ k_0 \})\implies \alpha \le \frac 12$. 
\end{itemize}
 
 \end{proposition} 

Our second goal is the study of the singularities (i.e. points of nondifferentiability) of $d$. In order to avoid  mixing up different effects (e.g. the presence of conjugate points), we limit our analysis to  constant coefficients  ($A(x)\equiv I$). In this setting,  singularities can only be generated  by the presence of the obstacle, indeed $d$ is smooth if $\mathcal{O}=\emptyset$. 

Recalling the definition of reachable gradients \eqref{eq:dstar}, we introduce the singular set of $d$  as follows
\begin{multline}\label{insiemes} 
\Sigma (d)=\{ x\in \overline{\R^n\setminus \mathcal{O}}\setminus\{ k_0\} \ :
\\ 
D^*d(x) \text{ has at least two elements}\}. 
\end{multline}
 It is well-known that $x\in \Sigma (d)\setminus \mathcal{O}$ if and only if $d$ fails to be differentiable at $x$ (see e.g. \cite{CS}).  

We begin to study the singular set by analyzing the "propagation" of singularities. 

 For this purpose, we need some preliminaries on  generalized gradient flows. Given a real valued function $u$ defined on an open set $\Omega$, for $x\in \Omega$ we define 
$$
D^+u(x)=\left \{ p\in \R^n \ | \ \limsup_{\Omega \ni y\to x} \frac{u(y)-u(x)-\langle p, y-x\rangle }{|y-x|}\le 0   \right\}.
$$
If $u$ is a semiconcave function with linear modulus, then we have that $D^+u(x)\not=\emptyset$, for every $x\in \Omega$. The generalized gradient flow of $u$ is given by 
\begin{equation}\label{eq:ggf}
\dot{x}(t)\in D^+u (x(t)),\qquad \text{ for a.e. }t\geq 0 .
\end{equation}
We observe that the existence of a solution of \eqref{eq:ggf}  is a classical result in the theory of
differential inclusions (see, e.g., \cite[p.98]{AuC}).

For every $x_0\in \Omega$, Equation \eqref{eq:ggf} admits a unique Lipschitz continuous solution $x(\cdot )$ satisfying $x(0)=x_0$ (uniqueness is a consequence of linear semiconcavity). We observe that the flow associated with \eqref{eq:ggf} may have stationary points. Let us also point out that a solution of \eqref{eq:ggf} is a priori defined only until the first time $t_*$ such that $x(t_*)\in \partial\Omega$. 

If we add the additional information that $u$ is a viscosity solution of the eikonal equation on $\Omega$, then it is well known that the singular set of $u$, $\Sigma (u)$, is invariant for the generalized gradient flow, i.e., 
$$x(0)\in \Sigma (u)\implies x(t)\in \Sigma (u)\,,\qquad\forall t\geq 0\,.$$
More precisely, either $x(\cdot )$ reaches $\partial \Omega$ in finite time or $x(t)\in\Sigma (u)$ for every $t\geq 0$ (provided that $x(0)\in \Sigma (u)$).  
For the proof of  the invariance of the singular set for short time see \cite{AC2}; invariance for all times is a consequence of the results in \cite{A5}.   

In the following, we assume that $d$ is singular at a point $x_0\notin \mathcal{O}$ and we obtain different conclusions depending on whether  $x_0\in \co \mathcal{O}$ or  $x_0\notin \co \mathcal{O}$.  
Hereafter, $\co \mathcal{O}$ stands for the convex hull of $\mathcal{O}$.
\begin{theorem}\label{t3}
Under Assumption (O), let $k_0\in \R^n\setminus \mathcal{O}$, assume that $A\equiv I$, and let $d$ given by \eqref{d}.

(i) Let  $x_0\in  \Sigma (d)\cap (\co \mathcal{O}\setminus \mathcal{O})$, with $x_0\not= k_0$. Then there exist   $\sigma>0$
and a Lipschitz map 
$$
[0,\sigma [\ni t\mapsto x(t)\in \Sigma (d) 
$$
such that $x(0)=x_0$ and  $x(t)\not=x_0$, for every $t\in ]0,\sigma [$.

(ii)  Let  $x_0\in \Sigma (d) \setminus \co \mathcal{O}$, with $x_0\not= k_0$, and let $x(\cdot )$ be the solution of the equation 
\begin{equation}\label{GF}
\begin{cases}
\dot{x}(t)\in D^+d(x(t)),\qquad t\geq 0,
\\
x(0)=x_0. 
\end{cases}
\end{equation} 
Then, we have that $x(t)\in \Sigma (d)$ for every $t\geq 0$, $t\mapsto d(x(t))$ is a strictly increasing function (i.e., $t\mapsto x(t)$ is an injective map), and 
$\{ x(t)\ | \ t\geq 0\}$ is an unbounded set. 
\end{theorem} 
\begin{remark}
We observe that there are several papers dealing with the (local) propagation of singularities for solutions of Hamilton-Jacobi equations when $x_0$ is not a critical point for $d$ (i.e., $0\notin D^+d(x_0)$), see, e.g., \cite{AC2}, \cite{CY} and \cite{CC}. Since this assumption is not satisfied in general, Theorem \ref{t3}(i) cannot be deduced from these results. 
For instance, one can easily find $\mathcal{O}\subset \R^2$ such that there exists $x_0\in \co\mathcal{O}\setminus \mathcal{O}$ with $0\in D^+d(x_0)$. In this special case, the singular curve provided by \eqref{GF} reduces to the constant arc $x(t)\equiv x_0$, for every $t\geq 0$, while 
the singular curve given by Theorem \ref{t3}(i) is not constant.  
 
Instead, Theorem \ref{t3}(ii) reduces to already established propagation results 
provided that 
\begin{itemize} 
\item no critical points of $d$ are present in the complement of $\co \mathcal{O}$, 
\item the solution of equation \eqref{GF} does not intersect $\mathcal{O}$ for positive times. 
\end{itemize}
Finally, let us point out that the presence of unbounded components of the singular set follows, in the case of solutions of evolutive Hamilton-Jacobi equations, from the result in \cite{A5}. Furthermore, in the case of stationary equations, the unbounded components of the singular set of the Euclidean distance from a closed set are analyzed in \cite{CP}. 
\end{remark}

For convex obstacles, we have propagation of singularities at infinity even from boundary points.   
\begin{theorem}[Convex obstacle]\label{t2}
Under Assumption (O),  let $\mathcal{O}$ be a convex set, let $k_0\in \R^n\setminus \mathcal{O}$, assume that $A\equiv I$, and let $d$ given by \eqref{d}.
Let  $x_0\in  \partial\mathcal{O}\cap \Sigma (d)$\footnote{We observe that,  by Theorem \ref{t:exsing},  we have that $\partial\mathcal{O}\cap \Sigma (d)\not=\emptyset$.}  and let $x(\cdot )$ be the solution of the equation 
$$
\begin{cases}
\dot{x}(t)\in D^+d(x(t)),\qquad t\geq 0,
\\
x(0)=x_0. 
\end{cases}
$$
Then, we have that $x(t)\in \Sigma (d)$ for every $t\geq 0$, $t\mapsto x(t)$ is an injective map, and $\{ x(t)\ | \ t\geq 0\}$ is an unbounded set. 
\end{theorem}
 In the general case we have a local propagation result. 
\begin{theorem}\label{t2nc}
Under Assumption (O), let $k_0\in \R^n\setminus \mathcal{O}$, assume that $A\equiv I$, and let $d$ given by \eqref{d}.
Let  $x_0\in  \partial\mathcal{O}\cap \Sigma (d)$.
Then there exist $\sigma>0$
and a map 
$$
[0,\sigma [\ni t\mapsto x(t)\in \Sigma (d) 
$$
such that $x(0)=x_0$, $\lim_{t\to 0^+}x(t)=x_0$, and  $x(t)\in\R^n\setminus\mathcal O$ for all $t\in ]0,\sigma [$. 
\end{theorem}
\begin{remark} 
As a consequence of the results above, for every singular point $x_0\not= k_0$ of the distance function, in the presence of an obstacle, there exists a continuum of singular points through $x_0$.  
\end{remark}
In the next results we study the existence of  singularities for the distance function for convex obstacles.
 
\begin{theorem}\label{t:exsing}
Under Assumption (O), let $\mathcal{O}$ be a convex set, let $k_0\in \R^n\setminus \mathcal{O}$, assume that $A\equiv I$, and let $d$ be given by \eqref{d}. Then,  
$$\partial \mathcal{O}\cap \Sigma (d)\not=\emptyset.$$
\end{theorem}  
The convexity assumption in Theorem \ref{t:exsing} is of technical nature. Understanding the case of a general compact obstacle is an open problem.

\section{Proofs} 
 
\subsection{Preliminaries on semiconcave functions} 

We recall the result about the extension of a semiconcave function with fractional modulus, obtained in \cite{ABC}. Such a result will be used several times in the sequel. 
 \begin{theorem}\label{t:estensione}
Let $\Omega\subset \R^n$ be an open set with boundary of class $C^{1,1}$, and let $u\in SC_{loc}^\alpha (\overline{\Omega})$. Then, for every $x\in \partial\Omega$ and there exist $\delta>0$ and a function $E(u)\in SC^\alpha (B_\delta (x))$ such that 
\begin{enumerate} 
\item $E(u)(y)=u(y)$ for every $y\in B_\delta (x)\cap \overline{\Omega}$;
\item $D^*E(u)(y)=D^*u(y)$ for every $y\in B_\delta (x)\cap \partial\Omega$.
\end{enumerate}
\end{theorem}
In particular, for the applications of interest to this paper, we will take $u=d$, $\alpha =1/2$ and $\Omega=\R^n\setminus\{ \mathcal{O}\cup \overline{B}_r(k_0)\}$, for a suitable positive $r$ less than the distance of $k_0$ from $\partial\mathcal{O}$.
\begin{remark} 
We observe that $x\in \Sigma (d)\cap \partial\mathcal{O}$ if and only if each local semiconcave extension of $d$ is not differentiable at $x$. 
 Indeed, if $D^*d(x)$ contains $p_0,p_1$, with $p_0\not= p_1$, then for every local extension of $d$, $\tilde{d}$, we have that $p_0,p_1\in D^*\tilde{d}(x)$, by the definition of   reachable gradients, i.e., $x\in \Sigma (\tilde{d})$.  Vice versa, if $x\in \partial\mathcal{O}\setminus \Sigma (d)$, then  Theorem~\ref{t:estensione} provides a local extension of $d$ which is differentiable at $x$.  
\end{remark} 
Let us also recall a result on the ``propagation'' of singularities for semiconcave functions with linear modulus (see \cite[Theorem 4.2]{AC1}). 
\begin{theorem}\label{t:acprop}
Let $\Omega\subset \R^n$ be an open set and let $u\in SC^1(\Omega)$. Let $x_0\in\Omega$ be such that 
$$
\partial D^+u(x_0)\setminus D^*u(x_0)\not=\emptyset . 
$$
Then, there exists a Lipschitz arc, 
$$
[0,\sigma [\ni t\mapsto x(t)\in \Sigma (u),
$$
 such that $x(0)=x_0$  and $x(t)\not= x_0$, for every $t\in ]0,\sigma [$. 
 \end{theorem}
Finally, we will need also the following propagation result (see \cite[Theorem 1.5]{ABC}). 
\begin{theorem}\label{t:abcprop}
Let  $\Omega$ be an open set with boundary of class $C^{1,1}$, let  $u\in SC^\alpha_{loc} (\overline{\Omega})$ and let $x_0\in \partial \Omega$ such that 
\begin{equation}\label{eq:h}
\partial \co D^*u(x_0)\setminus D^*u(x_0)\not=\emptyset .
\end{equation}
Let $p_0\in \co D^*u(x_0)\setminus D^*u(x_0)$ and let $-\theta$ be a vector in the normal cone to $\co D^*u(x_0)$ at $p_0$. 
Let $E(u)\in SC^\alpha (B_\delta (x_0))$ be an extension of $u$ satisfying property $(2)$ of Theorem~\ref{t:estensione}.   
Then, there is a map 
$$
[0,\sigma ]\ni s\mapsto x(s)\in B_\delta (x_0)
$$
(depending on $E(u)$)  such that 
\begin{enumerate}
\item $x(0)=x_0$ and $\lim_{s\to 0^+} x(s)=x_0$,  
\item   $x(s)\not=x_0$, for every $s\in [0,\sigma ]$;
\item $x(s)\in \Sigma (E(u))$, for every $s\in [0,\sigma ]$;
\item $x(s)=x_0+s\theta  +o(s)$ with $o(s)/s\to 0$ as $s\to 0^+$, 
\end{enumerate}
for a suitable $\sigma>0$ depending on the "initial" point $x_0$.  
\end{theorem}
In order to study the ``propagation'' of singularities from a point on the boundary of a nonconvex obstacle, we need the following  
\begin{theorem}\label{t:fleming} 
Let $\Omega \subset \R^n$ be an open set with a $C^2$-boundary and let $u\in SC^\alpha_{loc}(\overline{\Omega})$, for a suitable $\alpha\in ]0,1]$.  
Let $x_0\in \Sigma (u)\cap \partial\Omega$ and let us suppose that there exists $p_0\in \co D^*u(x_0)\setminus D^*u(x_0)$ such that 
\begin{equation}\label{eq:exitdirection} 
p_0+t\nu (x_0)\notin \co D^*u(x_0), \qquad \forall t>0.
\end{equation}
(Here $\nu (x_0)$ is the exterior unit normal to $\partial \Omega$ at $x_0$.)  
Then, we can find $T >0$ and an arc $x:[0,T]\longrightarrow \Sigma (u)\cap \overline{\Omega}$ such that 
\begin{equation}\label{eq:px}
\begin{cases}
\lim_{t\to 0^+}x(t)=x(0)=x_0,
\\
\Omega \ni x(t)\not= x_0, \qquad \forall t\in ]0,T].
\end{cases}
\end{equation}
\end{theorem} 
\begin{remark} 
Theorem \ref{t:fleming} is based on a method introduced  in \cite{AC0}. We observe that the conclusion in Theorem \ref{t:fleming} is slightly weaker than the one in Theorem \ref{t:abcprop}. Indeed, in Theorem  \ref{t:fleming},   a ``direction of propagation'' is not singled out while this is done in (4) of Theorem \ref{t:abcprop}. On the other hand, the advantage of Theorem \ref{t:fleming} consists in the fact that Assumption \eqref{eq:exitdirection} is weaker than the one of Theorem \ref{t:abcprop}. In other words, in order to apply Theorem \ref{t:fleming}, we need less information on the exposed faces of the convex set $\co D^*u(x_0)$. 
\end{remark}
In order to prove Theorem \ref{t:fleming}, we need the following 
\begin{theorem}\label{t:acfleming} 
Let $g\in SC^\alpha (B_r(0))$ for some $r>0$ and $\alpha \in ]0,1]$, and let $0\in \Sigma (g)$. 
Let 
\begin{equation}\label{12}
p_0\in D^+g(0)\setminus D^*g(0), 
\end{equation}
and suppose that for some vector $q\in \R^n\setminus \{ 0\}$ 
\begin{equation}\label{13} 
p_0+tq\notin D^+g(0),\qquad \forall t>0. 
\end{equation}
Then, a number $T>0$ and an arc $x:[0,T]\longrightarrow B_r(0)$ exist so that 
\begin{itemize}
\item[$(i)$] $|x(t)|^{1+\alpha }<-t\langle x(t) ,q\rangle $, for every $t\in ]0,T]$, 
\item[$(ii)$]  $x(0)=0$ and $|x(t)|< (|q| t)^{1/\alpha }$, for every $t\in ]0,T]$, 
\item[$(iii)$]  $x(t)\in \Sigma (g)$, for every $t\in [0,T]$. 
\end{itemize}  
\end{theorem}  
The proof of Theorem \ref{t:acfleming} follows the same lines of the one of \cite[Theorem 4.1]{AC0}. We provide such a proof for the reader convenience. 
\begin{proof} 
Without loss of generality, possibly reducing $r$,  we may assume that $g(0)=0$ and 
\begin{equation}\label{eq:supdif}
g(x)-\langle p_0,x\rangle -C |x|^{1+\alpha }\le 0,\qquad  \forall x\in B_r(0),  
\end{equation}
for a suitable $C>0$. Set 
$$
\phi (x):=g(x) -(C+1) |x|^{1+\alpha }
$$
Notice that $\phi \in SC^\alpha (B_r(0))$ and $D^+g(0)=D^+\phi (0)$. 
Define 
$$
\psi_t (x)= \phi (x) -\langle p_0+tq , x\rangle ,\qquad x\in \bar{B}_r(0), 
$$
and observe that,  due to  \eqref{13} and the semiconcavity of $g$, we can find $y_t\in B_r(0)$ such that 
\begin{equation}\label{eq:yt}
\psi_t (y_t)>0 , 
\end{equation} 
for every $t>0$. Then, consider the maximum of the function $\psi_t$ over the ball $\bar{B}_r(0)$ and let $x(t)\in \bar{B}_r(0)$ so that 
$$
\max_{ x\in \bar{B}_r(0)} \psi_t (x)=\psi_t (x(t)), \qquad (t>0). 
$$
We define $x(0)=0$ 
 Now, \eqref{eq:yt} yields that $\psi_t (x(t))>0$ so, by \eqref{eq:supdif}, we deduce that 
 $$
 0< -t \langle q,x\rangle - |x|^{1+\alpha} ,\qquad \forall t>0.
 $$
So (i) and (ii) follow. In order to complete the proof of Theorem \ref{t:acfleming} it remains to show that there is $T_0>0$ such that $x(t)\in \Sigma (g)$ for every $t\in [0,T_0]$. We note that for, every $t>0$ sufficiently small, $x(t)\in B_r(0)$ and so 
$$
0\in D^+g(x(t)).
$$
Then, a direct computation yields that 
\begin{equation}\label{eq:pt} 
p_0+tq+ (1+\alpha )(C+1) |x(t)|^{\alpha -1}\,  x(t)\in D^+g(x(t)),   
\end{equation}
for every $t>0$ sufficiently small.  
We claim that there exists $T>0$ such that $x(t)\in \Sigma (g)$ for every $t\in (0,T]$. Let us argue by contradiction assuming that there is a positive sequence $t_k$, converging to $0$ such that $x(t_k)\notin \Sigma (g)$. Then, in view of \eqref{eq:pt}, we have that 
$$
Dg(x(t_k))= p_0+t_k q+ (1+\alpha )(C+1) |x(t_k)|^{\alpha -1}\,  x(t_k), 
$$
and, taking the limit as $k\to \infty$ in the formula above, we find that $p_0\in D^* g(0)$ in contrast with Assumption \eqref{12}.  This proves Conclusion (iii).
\end{proof} 
Now, we are ready to prove Theorem \ref{t:fleming} 

\begin{proof} 
Let us denote by $E(u)$ the semiconcave extension of $u$ provided by Theorem \ref{t:estensione}. Then, we can apply Theorem \ref{t:acfleming} to $E(u)$, defined on $B_\delta(x_0)$ for some $\delta>0$. In order to complete the proof, we only need to show that, possibly reducing $T$,  $x(t)\in \Omega$ for every $t\in ]0,T]$. Owing to the $C^2$ regularity of $\partial \Omega$, we can find $R>0$ such that 
$$
\left\{ y\in \R^n \ | \ \frac 1{2R} |y-x_0|^2< -\langle \nu (x_0) , y-x_0\rangle \right \} \subset \Omega . 
$$
Furthemore, Theorem \ref{t:acfleming} (i) yields that 
$$
\frac 1 t |x(t)-x_0|^{1+\alpha } <   -\langle \nu (x_0) , x(t)-x_0\rangle  ,\qquad \forall t\in ]0,T]. 
$$
Then, since  $t |x(t)-x_0|^2< 2R  |x(t)-x_0|^{1+\alpha } $ for $0<t<   2R/\delta^{1-\alpha} $,  we conclude that, possibly reducing $T$, $x(t)\in \Omega$ for every $t\in ]0,T]$. This completes our proof.  
\end{proof} 

\subsection{Preliminaries on length minimizers} 

For $x\in  \overline{\R^n\setminus \mathcal{O}}$, we define 
$$ 
\Gamma^* [x]=\{ \gamma\in \Gamma [x] \ | \    \tau (\gamma )=\inf_{\Gamma [x] } \tau \}.
$$
In other words, $\Gamma^* [x]$ is the set of all the length minimizers joining $x$ with $k_0$. 
If $A\equiv I$,  it is clear that any shortest path consists of straight-line segments and curves of minimal length on the boundary of the obstacle. 
Then,  since $\partial\mathcal{O}\in C^{2}$, any length minimizer is piecewise $C^2$ and it is (globally) $C^{1,1}$. 

\begin{remark} \label{r:geo} 
The $C^{1,1}$ regularity of length minimizers is a well-kown fact. 
To our knowldge,  $C^1$ regularity was established in \cite{AA} for $C^3$-Riemannian manifold with $C^1$-boundary. 
Furthermore, in the case of an obstacle with $C^3$ boundary,  the fact that   length minimizers  have first and second derivatives in $L^2$ follows from the results in \cite{MS}. 
Finally, in the case of an obstacle with a $C^2$ boundary, the $C^{1,1}$ regularity of  length minimizers is a consequence of the results given in \cite{Ca} (see also  \cite{ER}). In particular, \cite[Theorem 3.2]{Ca} ensures that, given a compact set $K\subset \R^n$, there exists a constant $c=c(K)$ such that for every $x\in K\cap (\overline{\R^n\setminus \mathcal{O}})$ and $\gamma \in \Gamma^*[x]$ we have that 
\begin{equation}\label{eq:bgeo}
\| \ddot{\gamma}\|_\infty \le c.
\end{equation}
\end{remark} 

\begin{remark} \label{r:boundaryreg} 
Observe that, if a length minimizer $\gamma$ touches $\mathcal O$ at an ``interior'' point $\gamma(t)\,,\;(t\in ]0,d(\gamma(0))[$), then $\dot{\gamma}(t)$ is tangent to $\mathcal O$ at $\gamma(t)$. This fact is a consequence of the  $C^{1,1}$ regularity of minimizers.
\end{remark} 
 
We now proceed  to relate the elements of $\Gamma^* [x]$ with suitable generalized gradients of $d$.  
We recall that if  $d\in SC_{loc}^\alpha  (  \overline{\R^n \setminus \mathcal{O}}\setminus \{ k_0\})$ for a suitable $\alpha\in ]0,1]$, as a consequence of  Rademacher's Theorem, we have that $D^*d(x)\not=\emptyset$,
for every $x\in \R^n \setminus (\mathcal{O}\cup \{ k_0\} )$.  
Furthermore, for every compact set $K\subset  \overline{\R^n \setminus \mathcal{O}}\setminus \{ k_0\} $, every $x,y\in K$ such that $[x,y]\subset K$, and every $p\in D^+d(x)$, we have that   
\begin{equation}\label{p}
d(y)\le d(x)+\langle p,y-x\rangle +C |y-x|^{1+\alpha}, 
\end{equation}
for a suitable constant $C$ depending on $K$. 
We point out that \eqref{p} is a consequence of the assumption  that $d\in SC_{loc}^\alpha  (  \overline{\R^n \setminus \mathcal{O}}\setminus \{ k_0\}))$ for a suitable $\alpha\in ]0,1]$. 
\begin{lemma}\label{l:vi}
Under Assumption (O), let $k_0\in \R^n\setminus \mathcal{O}$, let $A\equiv I$,  and let $d$ be given by \eqref{d}. Then, for every $x\in (\overline{\R^n\setminus \mathcal{O}})\setminus \{ k_0\}$ and  $\gamma \in \Gamma^*[x]$, we have that 
\begin{equation}\label{eq:s}
-\dot{\gamma}(t)\in D^*d(\gamma (t)),\qquad \forall t\in [0,d(x)].
\end{equation}
Furthermore\footnote{Hereafter,  even at a point $y\in\partial\mathcal O$, we have kept the notation  $Dd(y)$ to denote the unique element of $D^*d(y)$ whenever the last set reduces to a singleton.},  
\begin{equation}\label{eq:rlt}
D^*d(\gamma (t))=\{ Dd(\gamma (t))\},\quad \forall t\in ]0,d(x)[. 
\end{equation} 
Finally, for every $x\in (\overline{\R^n\setminus \mathcal{O}})\setminus \{ k_0\}$ and for every $p\in D^*d(x)$ there exists $\gamma \in \Gamma^*[x]$ such that  $-\dot{\gamma}(0)=p$.
\end{lemma} 
Here and in the sequel we tacitly assume that $\dot{\gamma}(0)$ is the right derivative while $\dot{\gamma }(d(x))$ stands for the left derivative, i.e. $\dot{\gamma}(0)=\lim_{t\to 0^+}\frac{\gamma (t)-\gamma (0)}t$ and $\dot{\gamma}(d(x))=\lim_{t\to d(x)^-}\frac{\gamma (d(x))-\gamma (t)}{d(x)-t}$. 
\begin{remark} 
Observe that Lemma \ref{l:vi} ensures the existence of an injective map 
$$
D^*d(x)\longrightarrow \Gamma^*[x] 
$$
which fails, in general, to be onto.
In particular, there can be more than one minimizer starting from $x$ with the same initial velocity. This branching phenomenon for length minimizers depends on the presence of an obstacle. Indeed, for unconstrained problems, the uniqueness of solutions to the Hamiltonian formulation of the Maximum Principle allows to construct a bijection between $\Gamma^*[x]$ and $D^*d(x)$ (see e.g. \cite[Theorems 6.4.9 and 8.4.14]{CS}).       
\end{remark} 
\begin{proof} 
We  show first that \eqref{eq:s} holds.
 For every $\gamma \in \Gamma^*[x]$, we have that $d(x)-t=d(\gamma (t))$, for every $t\in [0,d(x)]$. We claim that 
\begin{equation}\label{eq:divdir}
-1=\min_{q\in \co D^*d(\gamma (t) )}\langle q, \dot{\gamma} (t)\rangle= \min_{q\in  D^*d(\gamma (t))}\langle q,\dot{\gamma} (t)\rangle =
\langle p_t,\dot{\gamma} (t)\rangle ,
\end{equation}
for every $t\in [0,d(x)[$ and for a suitable $p_t\in D^*d(\gamma (t))$. Indeed, for all $t\in [0,d(x)[$ such that $\gamma (t)\notin \mathcal{O}$, \eqref{eq:divdir} follows from \cite[Theorem~3.3.6]{CS}   and the fact that  $\Gamma^*[x]\subset C^{1,1}([0,d(x)];\R^n)$.
(We observe that the point $t=d(x)$ is excluded because
$\gamma (d(x))=k_0$ and $d(\cdot )$ is not semiconcave on a neighborhood of $k_0$.)
On the other hand, for all $t\in [0,d(x)[$ such that $\gamma (t)\in \mathcal{O}$, one can repeat the same argument as above 
using the extension of $d$ given by Theorem \ref{t:estensione} instead of $d$ and recalling that $D^*E(d)(\gamma (t))=D^*d(\gamma (t))$. 
  
Then, since $|\dot{\gamma} (t)|\le 1$ and $|p_t|=1$, we conclude that $p_t=-\dot{\gamma}(t)$. Thus,  \eqref{eq:s} holds for every $t\in [0,d(x)[$. 
This fact and the inclusion $\Gamma^*[x]\subset C^{1,1}([0,d(x)];\R^n)$ yield that also $\dot{\gamma} (d(x))\in D^*d(k_0)$ (we point out that, because of the Lipschitz continuity of around $k_0$, the set $ D^*d(k_0)$ is well-defined).
This completes the proof of \eqref{eq:s}. 

Now, let us show that \eqref{eq:rlt} hold true. 
For every $x\in (\overline{\R^n\setminus \mathcal{O}})\setminus \{ k_0\}$, $\gamma \in \Gamma^*[x]$, $t\in ]0,d(x)[$, and  $p\in D^*d(\gamma (t))$ by Theorem \ref{t:rb} and \eqref{p} we have that 
$$
d(\gamma (t-h))\le d(\gamma (t))+\langle p, \gamma (t-h)-\gamma (t)\rangle +C |\gamma (t-h)-\gamma (t)|^{\frac 32} ,
$$
for every $h>0$ suitably small.  We recall that $d(\gamma (t))=d(x)-t$ and $d(\gamma (t-h))=d(x)-t+h$. Then, we find 
$$
h\le \langle p, \gamma (t-h)-\gamma (t)\rangle +C |\gamma (t-h)-\gamma (t)|^{\frac 32}\,.
$$
So, dividing both the sides of the inequality above by $h$ and taking the limit as $h\to 0^+$, we conclude that 
$$
1\le \langle p, -\dot{\gamma} (t)\rangle .
$$
(Here we use the fact that $\gamma$ is differentiable at $t$.)
Then, recalling that $|p|=1$ and $|\dot{\gamma} |\le 1$, we conclude that $\dot{\gamma }(t)=-p$. Consequently,   $D^*d(\gamma (t))=\{ -\dot{\gamma}(t)\}$. This implies that $d$ is differentiable along $\gamma$ and $\dot{\gamma}(t)=-Dd(\gamma (t))$, for every $t\in ]0,d(x)[$, by possibly appealing to the extension of $d$ given by Theorem \ref{t:estensione} as above. 
Notice that the endpoints are always excluded.

Finally, to prove that every element of $D^*d(x)$ can be taken as the initial velocity of a length minimizer,
let $p\in D^*d(x)$ and let $x_h$ be a sequence converging to $x$ such that $Dd(x_h)\to p$ as $h\to \infty$. 
Then, there exists $\gamma_h\in \Gamma^*[x_h]$ such that $\dot{\gamma}_h(0)=-Dd(x_h)$. Hence, possibly taking a subsequence, we have that $\gamma_h$ uniformly converges to a limit $\gamma \in \Gamma ^*[x]$. Therefore, by the bound on the second derivatives given in \eqref{eq:bgeo}, we may assume that, up to a subsequence,  
\begin{equation}\label{eq:crlt}
\lim_{h\to\infty} \| \dot{\gamma}_h-\dot{\gamma} \|_\infty =0. 
\end{equation}
 Now, let $s_h\in ]0,\min\{ d(x),d(x_h)\}]$ be a decreasing sequence converging to $0$ such that  
\begin{enumerate} 
\item   $|\dot{\gamma}_h (s_h)- \dot{\gamma} (s_h)|<\frac 1h$ (this can be achieved because of \eqref{eq:crlt});
\item $|\dot{\gamma}_h(s_h)-\dot{\gamma}_h(0)|<\frac 1h$ (here we use the fact that $\Gamma^*\subset C^{1,1}$),   
\end{enumerate}  
for every $h\in \N$.  Then, using once more the inclusion $\Gamma^* \subset C^{1,1}$, (1) and (2) above, 
 we conclude  that 
$$
\dot{\gamma}(0)=\lim_{h\to \infty } \dot{\gamma}(s_h)=\lim_{h\to \infty } \dot{\gamma}_h(s_h)=\lim_{h\to \infty }\dot{\gamma}_h (0)=-\lim_{h\to \infty }Dd(x_h)=-p.
$$
This completes our proof.  
\end{proof} 

\subsection{Proof of Proposition \ref{cm} (i)} 
In this proof, we use a construction inspired by an example given in \cite[p.1019]{CM}.

Let $x\in S(k_0)\cap \partial \mathcal{O}$ and assume by, contradiction, that $d$ belongs to $SC^1(\overline{B_\delta (x_0)\setminus \mathcal{O}})$ for some $\delta >0$. 
Suppose, in addition, that $D^*d(x) $ is a singleton, that is, $D^*d(x)=\{p\}$. This implies no loss of generality since,  as we shall see in what follows, points satisfying such a property can be found arbitrarily close to any point in $S(k_0)\cap \partial \mathcal{O}$.                 

Consider the extension $E(d)$ given by Theorem~\ref{t:estensione}. By possibly reducing $\delta>0$, we may assume that $E(d)$ is defined on $B_\delta (x)$. Moreover,
\begin{equation*}
D^*E(d)(x)=\{p\}\,.
\end{equation*}
Let $C$ be a linear semiconcavity constant for $E(d)$ on $B_\delta (x)$. Then, 
\begin{enumerate} 
\item $p\not= 0$ (in fact, $|p|=1$ since $d$ is a solution of the eikonal equation); 
\item $E(d)(y)\le E(d)(x)+\langle p,y-x\rangle +C |y-x|^2\,,\;\forall\,y\in B_\delta (x)$, in view of  \eqref{p}.
\end{enumerate} 
We observe that, by (1) and (2) above, 
\begin{multline}\label{eq:strinc}
E(d)(y)\le E(d)(x)+\langle p,y-x\rangle +C |y-x|^2
\\
=E(d)(x)+ C \left ( \left | y-x+\frac{p}{2C}  \right |^2-\frac 1{4C^2}  \right ). 
\end{multline}
Hence,  we have that 
\begin{equation}\label{uniff}
E(d)(y)\le E(d)(x), \quad \forall y\in \overline{B}_{\frac 1{2C}}\left ( x-\frac{p}{2C} \right )\cap B_\delta (x).
\end{equation} 
Let us consider the curve $\gamma:[0,d(x)]\longrightarrow \overline{\R^n\setminus \mathcal{O}}$ such that $\gamma (0)=x$,  $\gamma (d(x))=k_0$ and $\dot{\gamma}(0)=-p$.  We point out that the existence of such a curve follows from Lemma \ref{l:vi}, together with the fact that $D^*d(\gamma(t))$ reduces to a singleton for all $t\in]0,d(x)[$. Moreover, 
\begin{equation}
\label{eq:s1}
\mbox{there exists}\; s>0\;\mbox{such that}\; \gamma (t)\in \partial \mathcal{O}\cap B_\delta (x)\,,\quad\forall t\in [0,s[
\end{equation}
(because otherwise $x\in I(k_0)$) and this justifies our additional assumption $D^*d(x)=\{p\}$.


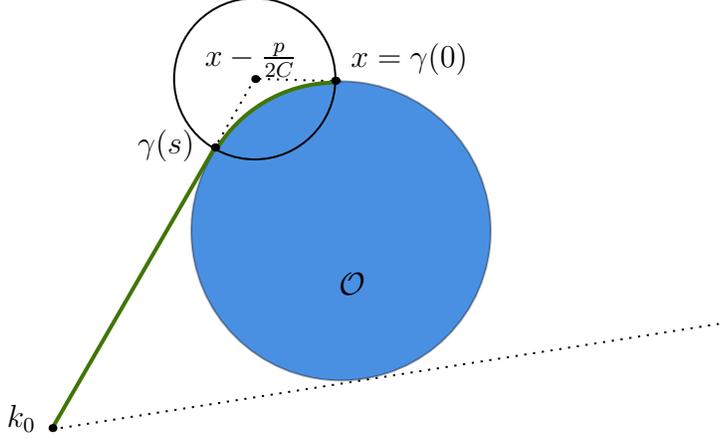
\tikzset{every picture/.style={line width=0.75pt}} 

\begin{tikzpicture}[x=0.75pt,y=0.75pt,yscale=-1,xscale=1]

\draw  [color={rgb, 255:red, 0; green, 0; blue, 0 }  ,draw opacity=0.34 ][fill={rgb, 255:red, 74; green, 144; blue, 226 }  ,fill opacity=1 ] (273.5,186.5) .. controls (273.5,144.8) and (307.3,111) .. (349,111) .. controls (390.7,111) and (424.5,144.8) .. (424.5,186.5) .. controls (424.5,228.2) and (390.7,262) .. (349,262) .. controls (307.3,262) and (273.5,228.2) .. (273.5,186.5) -- cycle ;
\draw  [dash pattern={on 0.84pt off 2.51pt}]  (305.5,109.75) -- (203,286.75) ;
\draw  [dash pattern={on 0.84pt off 2.51pt}]  (544,232.75) -- (203,286.75) ;
\draw  [dash pattern={on 0.84pt off 2.51pt}]  (349,111) -- (305.5,109.75) ;
\draw  [color={rgb, 255:red, 0; green, 0; blue, 0 }  ][line width=3] [line join = round][line cap = round] (285.46,144.49) .. controls (285.46,144.49) and (285.46,144.49) .. (285.46,144.49) ;
\draw  [color={rgb, 255:red, 0; green, 0; blue, 0 }  ][line width=3] [line join = round][line cap = round] (346.11,110.71) .. controls (346.11,110.71) and (346.11,110.71) .. (346.11,110.71) ;
\draw  [color={rgb, 255:red, 0; green, 0; blue, 0 }  ][line width=3] [line join = round][line cap = round] (305.75,109.85) .. controls (305.75,109.85) and (305.75,109.85) .. (305.75,109.85) ;
\draw   (264.82,109.75) .. controls (264.82,87.28) and (283.03,69.07) .. (305.5,69.07) .. controls (327.97,69.07) and (346.18,87.28) .. (346.18,109.75) .. controls (346.18,132.22) and (327.97,150.43) .. (305.5,150.43) .. controls (283.03,150.43) and (264.82,132.22) .. (264.82,109.75) -- cycle ;
\draw [color={rgb, 255:red, 65; green, 117; blue, 5 }  ,draw opacity=1 ][line width=1.5]    (285.5,144.25) -- (203,286.75) ;
\draw  [draw opacity=0][line width=1.5]  (286.69,144.59) .. controls (292.33,135.65) and (299.96,127.91) .. (309.44,122.09) .. controls (320.78,115.12) and (333.38,111.76) .. (346.03,111.66) -- (351.5,190.56) -- cycle ; \draw  [color={rgb, 255:red, 65; green, 117; blue, 5 }  ,draw opacity=1 ][line width=1.5]  (286.69,144.59) .. controls (292.33,135.65) and (299.96,127.91) .. (309.44,122.09) .. controls (320.78,115.12) and (333.38,111.76) .. (346.03,111.66) ;
\draw  [color={rgb, 255:red, 0; green, 0; blue, 0 }  ][line width=3] [line join = round][line cap = round] (203.5,286.03) .. controls (203.5,286.03) and (203.5,286.03) .. (203.5,286.03) ;
\draw  [color={rgb, 255:red, 0; green, 0; blue, 0 }  ][line width=3] [line join = round][line cap = round] (285.5,144.37) .. controls (285.5,144.37) and (285.5,144.37) .. (285.5,144.37) ;
\draw  [color={rgb, 255:red, 0; green, 0; blue, 0 }  ][line width=3] [line join = round][line cap = round] (346.17,110.7) .. controls (346.17,110.7) and (346.17,110.7) .. (346.17,110.7) ;

\draw (179,273) node [anchor=north west][inner sep=0.75pt]    {$k_{0}$};
\draw (346.5,206) node [anchor=north west][inner sep=0.75pt]    {$\mathcal{O}$};
\draw (245,134) node [anchor=north west][inner sep=0.75pt]    {$\gamma ( s)$};
\draw (352.5,90) node [anchor=north west][inner sep=0.75pt]    {$x=\gamma ( 0)$};
\draw (279,90) node [anchor=north west][inner sep=0.75pt]    {$x-\frac{p}{2C}$};

\end{tikzpicture}

Now, for every $(r,t)\in [0,s[\times [0,d(x)]$, we consider the family of curves $\Gamma(r,\cdot)$ defined as 
$$
\Gamma(r,t)=
\begin{cases}
\gamma (t),\quad &\text{ if }t\in [r,d(x)],
\\
\gamma (r)+ \dot{\gamma}(r)(t-r),\quad &\text{ if }t\in [0,r[.
\end{cases}
$$
%
We claim that $d(\Gamma (r,0))=d(\gamma (r)-r \dot{\gamma}(r))=d(x)$ and, consequently,
\begin{equation}\label{eq:dgr0}
E(d)(\Gamma (r,0))=d(\Gamma (r,0))=d(x)=E(d)(x),\qquad \forall r\in [0,s[.
\end{equation}
Indeed, $\Gamma (r,\cdot )$ is obtained by glueing together a part of a length minimizer $\gamma (\cdot )$ with  a straight-line segment, $t\mapsto \gamma (r)+ \dot{\gamma}(r)(t-r)$. More precisely, since $x\in S(k_0)\cap \partial \mathcal{O}$, we have also that $\Gamma (r,0)\in S(k_0)$ (for our purposes it suffices that $\Gamma (r,0)\in S(k_0)$ for $r$ near $0$, and this is true by continuity since $\Gamma (0,0)=x\in S(k_0)$).  
Then, a length minimizer with initial point at $\Gamma (r,0)$ can be decomposed into three parts 
\begin{enumerate}
\item a straight-line segment (tangent to the obstacle) joining the point $\Gamma (r,0)$ with $\gamma (r)\in \mathcal{O}$;
\item  a geodesic on the obstacle, $\gamma (\cdot )$, joining the tangency point $\gamma (r)$ with a point $\gamma (s)\in I(k_0)$ (and minimizing the distance on the obstacle between $\gamma (r)$ and  $I(k_0)$);
\item a straight-line segment joining $\gamma (s)$ with $k_0$. 
\end{enumerate}
Now, let us recall that the obstacle is a ball and that $\Gamma (r,0)$ is a point in the exterior of such a ball. 
Let us  consider the cone with vertex at $\Gamma (r,0)$ and tangent to the sphere; then the length of the parts of generators joining the vertex with the tangency points is a constant. Hence, $\Gamma (r,\cdot )$ is a length minimizer and  \eqref{eq:dgr0} follows. 

We observe that, for simmetry reasons, we may assume that for every $r\in [0,s[$ the curve $\Gamma (r,\cdot )$ lies on the plane of dimension $2$ containing the target $k_0$, the origin (which is the center of the ball $\mathcal{O}$) and the point $x$. 
   
Now, we claim that 
\begin{equation}\label{c}
\Gamma (r,0)=\gamma (r)- \dot{\gamma}(r)r\in \overline{ \R^n\setminus \left (\mathcal{O}\cup B_{\frac 1{2C}}\left ( x-\frac{p}{2C} \right )\right )}
\end{equation}
for every $r\in [0,s[$.  Indeed, recalling  \eqref{eq:strinc} we have that 
$$
y\in B_{\frac 1{2C}}\left ( x-\frac{p}{2C} \right )\cap (\overline{\R^n\setminus \mathcal{O}})\implies  d(y)< d(x) \,,
$$
while we have already noted that, for every $r\in [0,s[$, $\Gamma (r,0))$ belongs to the level set 
\begin{equation*}
\Lambda(x):=\big\{ y\in B_\delta (x)~|~E(d)(y)=E(d)(x)\big\}\,.
\end{equation*}
Incidentally, we note that the nonsmooth implicit function theorem (see, e.g., \cite[Section~7.1]{C}) ensures that $\Lambda(x)$ is a Lipschitz hypersurface near $x$, which is in fact differentiable at $x$
because $D^*E(d)(x)=\{p\}$.
%

Next, we define the curve 
$$
[0,s[\ni r\mapsto c(r):=\gamma (r)- \dot{\gamma}(r)r  , \quad \text{ with }c(0)=x , 
$$
and we note that   
\begin{itemize}
\item[(A)] $c(r)\in\Lambda(x)$, for every $r\in [0,s[$;
\item[(B)] $c(r)\in \overline{ \R^n\setminus \left (\mathcal{O}\cup B_{\frac 1{2C}}\left ( x-\frac{p}{2C} \right )\right )}$,   for every $r\in [0,s[$;
\item[(C)] the sphere $\partial B_{\frac 1{2C}}\left ( x-\frac{p}{2C} \right )$ is tangent to $\Lambda(x)$ at $x$.  
\end{itemize}
 Hence (A), (B) and (C) above imply that the curve $c(\cdot )$ is tangent to the sphere $\partial B_{\frac 1{2C}}\left ( x-\frac{p}{2C} \right )$ at $x$.

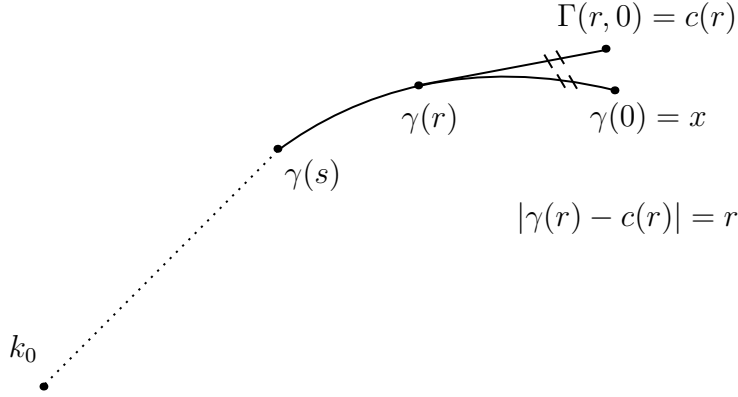
\tikzset{every picture/.style={line width=0.75pt}} 

\begin{tikzpicture}[x=0.75pt,y=0.75pt,yscale=-1,xscale=1]

\draw    (299,121.75) .. controls (339,91.75) and (397,73.75) .. (469,91.25) ;
\draw  [color={rgb, 255:red, 0; green, 0; blue, 0 }  ][line width=3] [line join = round][line cap = round] (298.9,120.8) .. controls (298.9,120.8) and (298.9,120.8) .. (298.9,120.8) ;
\draw  [color={rgb, 255:red, 0; green, 0; blue, 0 }  ][line width=3] [line join = round][line cap = round] (369.9,88.8) .. controls (369.9,88.8) and (369.9,88.8) .. (369.9,88.8) ;
\draw  [color={rgb, 255:red, 0; green, 0; blue, 0 }  ][line width=3] [line join = round][line cap = round] (180.9,240.8) .. controls (180.9,240.8) and (180.9,240.8) .. (180.9,240.8) ;
\draw  [color={rgb, 255:red, 0; green, 0; blue, 0 }  ][line width=3] [line join = round][line cap = round] (468.9,91.3) .. controls (468.9,91.3) and (468.9,91.3) .. (468.9,91.3) ;
\draw  [color={rgb, 255:red, 0; green, 0; blue, 0 }  ][line width=3] [line join = round][line cap = round] (464.4,70.3) .. controls (464.4,70.3) and (464.4,70.3) .. (464.4,70.3) ;
\draw  [dash pattern={on 0.84pt off 2.51pt}]  (299,121.75) -- (181,240.25) ;
\draw    (465.5,70.75) -- (370,88.75) ;
\draw    (440,83.25) -- (444,90) ;
\draw    (446,83.75) -- (450,90.5) ;
\draw    (439.5,71.25) -- (443.5,78) ;
\draw    (433.5,73.25) -- (437.5,80) ;

\draw (162,213.5) node [anchor=north west][inner sep=0.75pt]    {$k_{0}$};
\draw (360,96.5) node [anchor=north west][inner sep=0.75pt]    {$\gamma ( r)$};
\draw (301,124.75) node [anchor=north west][inner sep=0.75pt]    {$\gamma ( s)$};
\draw (455,96) node [anchor=north west][inner sep=0.75pt]    {$\gamma ( 0) =x$};
\draw (438.5,45.5) node [anchor=north west][inner sep=0.75pt]    {$\Gamma ( r,0) =c( r)$};
\draw (418,146.5) node [anchor=north west][inner sep=0.75pt]    {$|\gamma ( r) -c( r) |=r$};

\end{tikzpicture}

\vspace{1cm}

\tikzset{every picture/.style={line width=0.75pt}} 

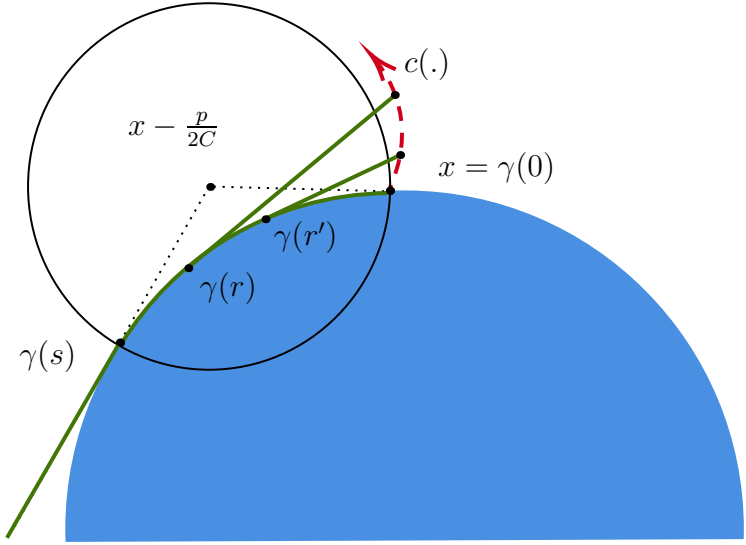
\begin{tikzpicture}[x=0.75pt,y=0.75pt,yscale=-1,xscale=1]

\draw  [color={rgb, 255:red, 74; green, 144; blue, 226 }  ,draw opacity=1 ][fill={rgb, 255:red, 74; green, 144; blue, 226 }  ,fill opacity=1 ] (586.74,293.49) .. controls (586.77,292.16) and (586.79,290.82) .. (586.8,289.47) .. controls (587.04,195.29) and (510.86,118.74) .. (416.63,118.49) .. controls (322.4,118.25) and (245.82,194.4) .. (245.57,288.58) .. controls (245.57,290.69) and (245.6,292.8) .. (245.67,294.89) -- cycle ;
\draw [color={rgb, 255:red, 208; green, 2; blue, 27 }  ,draw opacity=1 ][line width=1.5]  [dash pattern={on 5.63pt off 4.5pt}]  (401.09,51.62) .. controls (409.65,62.1) and (423.6,86.09) .. (407.54,119.34) ;
\draw [shift={(399,49.2)}, rotate = 47.73] [color={rgb, 255:red, 208; green, 2; blue, 27 }  ,draw opacity=1 ][line width=1.5]    (14.21,-4.28) .. controls (9.04,-1.82) and (4.3,-0.39) .. (0,0) .. controls (4.3,0.39) and (9.04,1.82) .. (14.21,4.28)   ;
\draw  [dash pattern={on 0.84pt off 2.51pt}]  (317.6,115.99) -- (219,288.5) ;
\draw  [dash pattern={on 0.84pt off 2.51pt}]  (410.33,118.5) -- (317.6,115.99) ;
\draw  [color={rgb, 255:red, 0; green, 0; blue, 0 }  ][line width=3] [line join = round][line cap = round] (272.62,194.98) .. controls (272.62,194.98) and (272.62,194.98) .. (272.62,194.98) ;
\draw  [color={rgb, 255:red, 0; green, 0; blue, 0 }  ][line width=3] [line join = round][line cap = round] (408.76,118.17) .. controls (408.76,118.17) and (408.76,118.17) .. (408.76,118.17) ;
\draw  [color={rgb, 255:red, 0; green, 0; blue, 0 }  ][line width=3] [line join = round][line cap = round] (318.16,116.22) .. controls (318.16,116.22) and (318.16,116.22) .. (318.16,116.22) ;
\draw   (226.29,115.99) .. controls (226.29,64.91) and (267.17,23.5) .. (317.6,23.5) .. controls (368.03,23.5) and (408.92,64.91) .. (408.92,115.99) .. controls (408.92,167.07) and (368.03,208.48) .. (317.6,208.48) .. controls (267.17,208.48) and (226.29,167.07) .. (226.29,115.99) -- cycle ;
\draw [color={rgb, 255:red, 65; green, 117; blue, 5 }  ,draw opacity=1 ][line width=1.5]    (272.7,194.43) -- (215.67,293.33) ;
\draw  [draw opacity=0][line width=1.5]  (273.79,194.66) .. controls (286.55,174.26) and (303.79,156.59) .. (325.19,143.28) .. controls (350.74,127.39) and (379.09,119.66) .. (407.54,119.34) -- (419.37,298.54) -- cycle ; \draw  [color={rgb, 255:red, 65; green, 117; blue, 5 }  ,draw opacity=1 ][line width=1.5]  (273.79,194.66) .. controls (286.55,174.26) and (303.79,156.59) .. (325.19,143.28) .. controls (350.74,127.39) and (379.09,119.66) .. (407.54,119.34) ;
\draw  [color={rgb, 255:red, 0; green, 0; blue, 0 }  ][line width=3] [line join = round][line cap = round] (272.7,194.7) .. controls (272.7,194.7) and (272.7,194.7) .. (272.7,194.7) ;
\draw  [color={rgb, 255:red, 0; green, 0; blue, 0 }  ][line width=3] [line join = round][line cap = round] (408.89,118.15) .. controls (408.89,118.15) and (408.89,118.15) .. (408.89,118.15) ;
\draw [color={rgb, 255:red, 65; green, 117; blue, 5 }  ,draw opacity=1 ][line width=1.5]    (412.2,69.2) -- (307.9,156.2) ;
\draw [color={rgb, 255:red, 65; green, 117; blue, 5 }  ,draw opacity=1 ][line width=1.5]    (412.6,100.8) -- (346.6,132) ;
\draw  [color={rgb, 255:red, 0; green, 0; blue, 0 }  ][line width=3] [line join = round][line cap = round] (413.96,100.12) .. controls (413.96,100.12) and (413.96,100.12) .. (413.96,100.12) ;
\draw  [color={rgb, 255:red, 0; green, 0; blue, 0 }  ][line width=3] [line join = round][line cap = round] (411.26,69.82) .. controls (411.26,69.82) and (411.26,69.82) .. (411.26,69.82) ;
\draw  [color={rgb, 255:red, 0; green, 0; blue, 0 }  ][line width=3] [line join = round][line cap = round] (307.16,157.22) .. controls (307.16,157.22) and (307.16,157.22) .. (307.16,157.22) ;
\draw  [color={rgb, 255:red, 0; green, 0; blue, 0 }  ][line width=3] [line join = round][line cap = round] (346.06,132.52) .. controls (346.06,132.52) and (346.06,132.52) .. (346.06,132.52) ;

\draw (220.46,192.32) node [anchor=north west][inner sep=0.75pt]    {$\gamma ( s)$};
\draw (431.7,97.27) node [anchor=north west][inner sep=0.75pt]    {$x=\gamma ( 0)$};
\draw (275.26,76.8) node [anchor=north west][inner sep=0.75pt]    {$x-\frac{p}{2C}$};
\draw (414.5,44.6) node [anchor=north west][inner sep=0.75pt]    {$c( .)$};
\draw (311.5,158) node [anchor=north west][inner sep=0.75pt]    {$\gamma ( r)$};
\draw (348.7,132.2) node [anchor=north west][inner sep=0.75pt]    {$\gamma ( r')$};

\end{tikzpicture}

Furthermore, we have that 
\begin{equation}\label{eq:cg}
\dot{c}(r)=- \ddot{\gamma}(r)r\qquad \text{ and }\qquad \ddot{c}(r)=- \dddot{\gamma}(r)r-\ddot{\gamma}(r).
\end{equation}
We observe that, since the curve $\gamma$ lies on a two dimensional plane, the curve $c(\cdot )$ lies on such a plane too, say  
$$
\Pi =\{O+ \lambda (c(0)-O)+\mu (c(s)-O)\ : \ \lambda ,\mu \in\R\}
$$
where $O$ is the center of the ball $\mathcal{O}$.  
Notice that $r\mapsto c(r)=\Gamma (r,0)$ is a smooth curve. Hence, by a well-known formula and \eqref{eq:cg},  the curvature of $c(\cdot )$ is given by  
$$
k(r)=\frac{|\dot{c}(r)\times \ddot{c}(r)|}{|\dot{c}(r)|^3}=\frac{|  (  - \ddot{\gamma}(r)r )  \times (- \dddot{\gamma}(r)r-\ddot{\gamma}(r))|}{|- \ddot{\gamma}(r)r |^3}
=\frac{|\ddot{\gamma}(r)\times \dddot{\gamma}(r)|}{r|\ddot{\gamma}(r)|^3},
$$
for $r\in ]0,s[$.  (Here ``$\times$'' denotes the vector product in the plane.) 

Moreover,  we find that, since $\gamma (r)$ ($r\in [0,s[$) is a geodesic arc on the sphere $\partial\mathcal{O}$ (it is given by the intersection of the plane $\Pi$ with the ball $\mathcal{O}$),
\begin{equation}\label{eq:rR}
k(r)=\frac 1{Rr},\qquad  r\in ]0,s[, 
\end{equation}
here $R$ is the radius of the obstacle $\mathcal{O}$. 
Indeed, by using in the two plane a system of coordinates such that $\gamma (0)=(R,0)$ and identifying $\R^2$  with $\C$,  we have that $\gamma (r)=R e^{ir}$, $\ddot{\gamma}(r)=-Re^{ir}$ and $\dddot{\gamma}(r)=-iRe^{ir}$, hence we find that $|\ddot{\gamma}(r)\times \dddot{\gamma}(r)|=R^2$ and $r|\ddot{\gamma}(r)|^3=rR^3$ and \eqref{eq:rR} follows.

Finally, recalling that  the planar curve $c(\cdot )$ is tangent to the intersection of  $\partial B_{\frac 1{2C}}\left ( x-\frac{p}{2C} \right )$ with $\Pi$ at $x$, we deduce that $k(r)$ is bounded above by the curvature of the sphere of radius $\frac{1}{2C}$, which leads to a contradiction 
with \eqref{eq:rR}.
%
This completes our proof. 

\subsection{Proof of Proposition \ref{cm} (ii)} 

Let $x=-\frac{k_0}{|k_0|}$ and  take  $\gamma \in \Gamma^*[x]$. Notice that $\gamma$ consists of two parts: an arc of a maximum circle on the sphere and a straight-line segment joining the endpoint of the arc closer to $k_0$ with $k_0$. Then, recalling that $\gamma$ is of class $C^{1,1}$ we reduce  the analysis to a subspace of dimension two. 
Specifically, we assume that $n=2$, $k_0=-M (1,0)$ (for a suitable $M>1$) and $x=(1,0)$ without loss of generality. 
By Lemma \ref{l:vi}, we have that  $-\dot{\gamma}(0)\in D^*d(1,0)$.  
Using the notation of the proof of Proposition \ref{cm}, by \eqref{p} we have that\footnote{$s$ is given by \eqref{eq:s1}.}, for all $r\in[0,s[$,
$$
d(c(r))\le d(1,0)+\langle -\dot{\gamma}(0), c(r)-(1,0)\rangle +C | c(r)-(1,0)|^{1+\alpha }. 
$$
%
Then, recalling that $d(c(r))= d(1,0)$, by the definition of $c(r)$ we find 
$$
\langle \dot{\gamma}(0), \gamma (r)-(1,0)-r\dot{\gamma}(r)\rangle \le C r^{1+\alpha} \left |   \frac{\gamma (r)-(1,0)}r -\dot{\gamma}(r)\right |^{1+\alpha}.
$$
Now, since $\gamma (r)=(\cos r,\sin r)$, we deduce that 
$$
\frac 1{r^\alpha} \left ( \frac{\sin r}r -\cos r\right ) \le C \left | \left ( \frac{\cos r -1}r+\sin r , \frac{\sin r} r -\cos r \right )\right |^{1+\alpha }
. 
$$
Then, up to higher order terms, we find that 
$$
\frac 13 r^{2-\alpha}\lesssim r^{1+\alpha }
$$
and taking the limit as $r\to 0^+$, we conclude that $1-2\alpha \geq 0$. This completes our proof. 
  
\subsection{Proof of Theorem \ref{t3} (i)}
 
 Let $x_0\in \Sigma (d) \cap [\co \mathcal{O}\setminus  \mathcal{O}]$. 
Then $B_{\delta} (x_0)\cap \mathcal{O}=\emptyset$ for  some $\delta >0$ and, by Theorem \ref{t:i}, we have that $d\in SC^1 (B_\delta (x_0))$. 
Recall that $ D^*d(x_0)\subset \partial D^+d(x_0)$ (see \cite[Proposition~4.4]{CiSo}). We claim that 
\begin{equation}\label{eq:ipotesi}
\partial D^+d(x_0)\setminus D^*d(x_0)\not=\emptyset .
\end{equation}
Indeed, let us assume by contradiction that $ \partial D^+d(x_0)= D^*d(x_0)$. Since $D^*d(x_0)\subset\partial B_1(0)$ and $D^+d(x_0)=\co D^*d(x_0)$, we have that  $D^*d(x_0)=\partial B_1(0)$. Then,  denoting by  
$\pi (x_0)$  the (Euclidean) projection of $x_0$ on $\partial \mathcal{O}$,  we find that   
$$
\nu (\pi (x_0))=\frac{x_0-\pi (x_0)}{| x_0-\pi (x_0)|}\in D^*d(x_0).
$$
Now,  Lemma \ref{l:vi} ensures the existence of a length minimizer $\gamma$ with
$$\gamma (t)=x_0 -t \nu (\pi (x_0) )\quad\forall t\in[0,|x_0-\pi (x_0)|]\,.$$
Then, on the one hand, $\gamma$ should be tangent to $\mathcal{O}$ at $\pi (x_0)$ owing to Remark~\ref{r:boundaryreg}. On the other hand, $\dot{\gamma}(t)=  -\nu (\pi (x_0) )$
for all $t\in[0,|x_0-\pi (x_0)|]$. This contradiction proves \eqref{eq:ipotesi}. 

Hence,  by Theorem \ref{t:acprop},  there exists a Lipschitz arc
$$
[0,\sigma [\ni t\mapsto x(t)\in \Sigma (d),
$$
 such that $x(0)=x_0$ and $x(t)\not= x_0$, for every $t\in ]0,\sigma [$. This completes our proof of Theorem \ref{t3}(i).

\subsection{Proof of Theorem~\ref{t3}~(ii)} 
We begin with the following 
\begin{lemma}\label{l:critical} 
Under Assumption (O), let $k_0\in \R^n\setminus \mathcal{O}$ and let $d$ be given by \eqref{d}. Then, $d$ has no critical points in $\R^n \setminus \co \mathcal{O}$, i.e. 
\begin{equation}\label{eq:0nid} 
0\notin  D^+d(x),\qquad \forall x\in \R^n \setminus \co \mathcal{O}\, .
\end{equation}
\end{lemma} 
Consequently, in the presence of a convex obstacle, $d$ has no critical points outside $\mathcal{O}$. 
\begin{proof}[Proof of Lemma \ref{l:critical}]
We begin with observing that \eqref{eq:0nid}  is trivial if $x\in I(k_0)$. So, we restrict to the case of $x\in S(k_0)$ that we analyse arguing by contradiction. Thus,  suppose that $0\in  D^+d(x)$ for some $x\in \R^n \setminus \co \mathcal{O}$. Then,  by Caratheodory's Theorem, there exist $\lambda_j\in [0,1]$ with $\sum_{j=1}^{n+1} \lambda_j =1$  and $p_j\in D^*d(x)$, $j=1,\ldots , n+1$, 
such that 
\begin{equation*}
0=\sum_{j=1}^{n+1} \lambda_j p_j.  
\end{equation*}
Furthermore, by Lemma~\ref{l:vi}, there exist arcs $\gamma_j\in\Gamma^*[x]$ touching the obstacle at points $x_j\in \partial \mathcal{O}$ such that 
$$
x_j=x-  |x_j-x|   p_j,\qquad j=1,\ldots , n+1 , 
$$
that is,
\begin{equation}\label{eq:conve} 
p_j= \frac{ x-x_j} {|x_j-x|},\qquad j=1,\ldots , n+1. 
\end{equation}
Hence, multiplying both  sides of \eqref{eq:conve} by $\lambda_j$ and taking the sum over $j=1,\ldots , n+1$, we deduce the contradiction  
$$
\co \mathcal{O}\not\ni x=\sum_{j=1}^{n+1} t_j x_j\in\co \mathcal{O} 
$$
with 
$$
t_j=\frac{\lambda_j/|x_j-x|}{ \sum_{\ell =1}^{n+1} \lambda_\ell /|x_\ell -x| },\quad  j=1,\ldots , n+1, 
$$
and \eqref{eq:0nid} follows. This completes our proof of Lemma \ref{l:critical}. 
\end{proof}

Let $x_0\in \Sigma (d)\cap (\R^n \setminus (\co \mathcal{O}\cup \{ k_0\} ))$ and consider the differential inclusion 
\begin{equation}\label{eq:gf} 
\begin{cases}
\dot{x}(t)\in D^+d(x(t)),\qquad \text{ for a.e. }t\geq 0,
\\
x(0)=x_0. 
\end{cases} 
\end{equation}
By \cite{AC2} (see also \cite{AC1}), there exists $\sigma >0$ and a unique Lipschitz continuous solution of \eqref{eq:gf}, such that 
$$
x(t)\in \Sigma (d),\qquad \forall t\in [0,\sigma [. 
$$
Furthermore, since $t\mapsto d(x(t))$ is a Lipschitz continuous function, for $\sigma>t_2>t_1\geq 0$ we have that   
\begin{equation}\label{eq:lip}
d(x(t_2))=d(x(t_1))+\int_{t_1}^{t_2} \frac{d}{ds}d(x(s))\, ds\,.
\end{equation}
Then, by the semiconcavity of $d$, for a.e. $s\in [t_1,t_2]$, we have that 
\begin{equation}\label{eq:semico}
\frac{d}{ds}d(x(s))=|p(s)|^2, 
\end{equation}
for a suitable $p(s)\in D^+d( x(s))$, see \cite[Theorem~1]{ACKS}. 
Hence, by \eqref{eq:lip}, \eqref{eq:semico}, and \eqref{eq:0nid} we deduce that 
\begin{equation}\label{eq:dtd0}
d(x(t_2))=d(x(t_1))+\int_{t_1}^{t_2} |p(s)|^2\, ds  > d(x (t_1)), 
\end{equation}
 i.e.
$
[0,\sigma [\ni t\mapsto d(x(t))
$
is an increasing function. 

Now, by the results in \cite{A5}, we can take $\sigma =+\infty$ provided that we can guarantee that $x(t)\notin \mathcal{O}$ for every $t\geq 0$.   
In fact, we will show that $x(t)\notin \co\mathcal{O}$ for every $t\geq 0$. 
Aiming at this, we observe that   $\forall t\in [0,\sigma [$
\begin{equation}\label{eq:xpiu}
\dot{x}^+(t):=\lim_{h\to 0^+} \frac{x(t+h)-x(t)}h= p(t)\in D^+d(x (t)) 
\end{equation}
where $p (t)$ is the minimizer of the function 
$$
D^+d(x (t))\ni p \mapsto |p|^2, 
$$
see  \cite[Corollary 3.4]{CY} for the proof of these facts. 
Let  $\text{ dist }( x(t), \co\mathcal{O})$  be the Euclidean distance of $x(t)$ from the set $\co\mathcal{O}$ and let $\pi (x (t))$ be the Euclidean projection of $x(t)$ onto  $\co \mathcal O$,
 i.e.,
 $$
 |x(t)-\pi(x(t))|=\min_{x\in \co\mathcal{O}}|x(t)-x|\,.
 $$
 Then
\begin{equation}\label{eq:nup}
\nu (\pi (x(t)))=\frac{x(t)-\pi(x(t))}{| x(t)-\pi(x(t))|}. 
\end{equation}
By the chain rule\footnote{We recall that the distance function from a convex set is differentiable in the complement of such a set.} and \eqref{eq:xpiu}, we find that 
\begin{multline}\label{eq:derdist}
\lim_{h\to 0^+} \frac{\text{ dist }( x(t+h),\co \mathcal{O})-\text{ dist }( x (t),\co  \mathcal{O}) } h 
\\
=\langle \nu (\pi (x (t))), p(t)\rangle .
\end{multline}
Now, let us show that the right hand side in \eqref{eq:derdist} is nonnegative. 
Recalling that $x(t)\in S(k_0)$ (since $x(t)\in \Sigma (d)$), we find  $x_1,\ldots ,x_{n+1}\in \co\mathcal{O}$ and $\lambda_j\in [0,1]$, with $\sum_{j=1}^{n+1}\lambda_j=1$, such that 
\begin{equation}\label{eq:pco}
p(t)=\sum_{j=1}^{n+1} 	\lambda_j \frac{x(t)-x_j}{|x(t)-x_j|}.     
\end{equation}
Then, since $\co \mathcal O$ is convex, we have that 
\begin{eqnarray*}
0&\geq& \langle x(t)-\pi (x(t)), x_j-\pi (x(t))\rangle 
\\
& =& \langle x(t)-\pi (x(t)), x_j-x(t)\rangle +| x(t)-\pi (x(t))|^2.  
\end{eqnarray*}
So, we find 
$$
 \langle x(t)-\pi (x(t)), x_j- x(t)\rangle \le 0 
$$
and, by \eqref{eq:pco} and \eqref{eq:nup}, we deduce  that 
$$
\langle \nu (\pi (x(t))), p(t)\rangle \geq 0.  
$$
In other words,  
$$
\lim_{h\to 0^+} \frac{\text{ dist }( x(t+h),\co \mathcal{O})-\text{ dist }( x (t),\co  \mathcal{O}) } h 
 \geq 0 
$$
for every $t\in [0,\sigma [$. Hence, since  $t\mapsto \text{ dist }( x (t),\co  \mathcal{O})$ is a Lipschitz continuous function, we conclude that, for $\sigma >t_2>t_1\geq 0$ 
\begin{multline*}
\text{ dist }( x (t_2),\co  \mathcal{O})=\text{ dist }( x (t_1),\co  \mathcal{O})+\int_{t_1}^{t_2} \frac d{dt} \text{ dist }( x (t),\co  \mathcal{O})\, dt 
\\
\geq \text{ dist }( x (t_1),\co  \mathcal{O}), 
\end{multline*}
i.e., $[0,\sigma [\ni t\mapsto \text{ dist }( x (t),\co  \mathcal{O})$  is 
a nondecreasing function. As a consequence, $\sigma =+\infty$.

In order to complete the proof, it remains to show that $\{ x(t)\ | \ t\geq 0\}$ is an unbounded set. 
We argue by contradiction: let us assume that 
$$
s:=\sup \{ d(x(t))\ | \ t\geq 0\} <+\infty .
$$
Then, there exists $t_j\to +\infty $ and $\bar{x}$ such that $x(t_j)\to \bar{x}$ and $s=d(\bar{x})$, as $j\to +\infty $. Hence, taking in 
\eqref{eq:dtd0} $t_2=t_j$ and $t_1=0$, we find that 
$$
d(x(t_j))=d(x_0)+\int_0^{t_j} |p(t)|^2\, dt \geq d(x_0) +t_j \inf_{t\in [0,t_j]}  | p(t)|^2
$$
and, sending $j\to \infty$, we deduce that 
$$
\limsup_{j\to \infty }\,   t_j  \inf_{t\in [0,t_j]}  | p(t)|^2     <+\infty . 
$$
So, necessarily, we have that  $\lim_{j\to \infty }  \inf_{t\in [0,t_j]}  | p(t)|^2  =0$. Now, by \eqref{eq:0nid}, $|p(t)|>0$ for every $t\geq 0$, and we would find a sequence $[0,t_j]\ni t^*_j\to \infty$ such that $p(t^*_j)\to 0$ as $j\to \infty $. Then, from the fact that the set valued map $x\mapsto D^+d(x)$ is upper semicontinuous and takes compact values, we find the contradiction  
$0\in D^+d(\bar{x})$ (we recall that   $\bar{x}\notin \co \mathcal{O}$ and, by \eqref{eq:0nid}, $0\notin  D^+d(\bar{x})$). It follows that $\{ x(t)\ | \ t\geq 0\}$ is an unbounded set, and this completes our proof of Theorem \ref{t3}(ii).

 \subsection{Proof of Theorem \ref{t2}}

Theorem~\ref{t2} is proved in two steps: first, appealing to the abstract propagation result  in \cite{ABC} (see also \cite{A1}), we show that a singularity on the boundary of a convex obstacle lies on a continuum of singular points which immediately enters the complement of $\mathcal O$. 
Then,  any point of such a continuum is the starting point of a singular Lipschitz curve owing to Theorem~\ref{t3}. So, we consider the limit curve, as the initial point converges to the singular point on the boundary of the obstacle, and we show that this curve satisfies all the requirements in Theorem \ref{t2}. 

\smallskip
Let $x_0\in \partial \mathcal{O}\cap \Sigma (d)$. Then  $x_0\in S(k_0)$ and $D^*d(x_0)$ has at least two elements. 
So, by Theorem~\ref{t:rb} above and Theorem~\ref{t:estensione}, there exist  $\delta >0$ and a function $E(d)\in SC^{\frac 12} (B_\delta (x))$ such that 
\begin{enumerate} 
\item $E(d)(y)=d(y)$ for every $y\in B_\delta (x)\cap \overline{\R^n\setminus \mathcal{O}}$;
\item $D^*E(d)(y)=D^*d(y)$ for every $y\in B_\delta (x)\cap \partial\mathcal{O}$.
\end{enumerate}
Moreover,  
\begin{equation}\label{eq:sc}
D^+E(d)(x_0)=\co D^*E(d)(x_0)\quad (=\co D^*d(x_0)). 
\end{equation}
Now, we need an orthogonality property of reachable gradients at $x$ with respect to the outward unit normal $\nu (x)$ to $\mathcal{O}$ at $x$.
\begin{lemma}\label{l:np}
Under Assumption (O), let $\mathcal{O}$ be a convex set. Then, for all $x\in \partial \mathcal{O}\cap S(k_0)$, we have that
\begin{equation*}
\langle p, \nu (x)\rangle = 0\,,\quad\forall \, p\in D^*d(x)\,.
\end{equation*}
\end{lemma}
\begin{proof}[Proof of Lemma \ref{l:np}]

Let us first show that, 
for every $x\in \partial \mathcal{O}\cap S(k_0)$, %
\begin{equation}\label{eq:diseq}
\langle p, \nu (x)\rangle \leq  0,\qquad \forall p\in D^*d(x). 
\end{equation}
 Fix  $x\in \partial \mathcal{O}\cap S(k_0)$ and let $p\in D^*d(x)$. By Lemma \ref{l:vi}, there exists $\gamma \in \Gamma^*[x]$ such that $\dot{\gamma}(0)=-p$. Then, $\langle \nu (x),\dot{\gamma}(0)\rangle \geq 0$ because  $\gamma (t)\in \overline{\R^n \setminus \mathcal{O}}$ for every $t\in [0,d(\gamma (0)]$, and \eqref{eq:diseq} follows. 
 
Now, let us show that equality holds in \eqref{eq:diseq}. We argue by contradiction assuming that
$$
\langle q, \nu (x)\rangle < 0,\qquad \text{ for a suitable } q\in D^*d(x). 
$$
Then, by Lemma \ref{l:vi}, we find that the curve $\gamma (t)=x-qt$ is a length minimizer up to some time $t_*>0$ such that 
\begin{equation}\label{eq:egua}
\begin{cases}
\gamma (t)\notin \mathcal{O},\qquad \forall t\in ]0,t_*[, 
\\
\gamma (t_*)\in\mathcal{O}. 
\end{cases}
\end{equation}
Then, by the convexity of $\mathcal{O}$ and the fact that $\gamma (t)\in\mathcal{O}$ for $t=0,t_*$, we obtain that $\gamma (t)\in \mathcal{O}$ for every $t\in ]0,t_*[$ (in contrast with \eqref{eq:egua}). Hence, we find that $\gamma (t)\notin \mathcal{O}$ for every $t>0$ (i.e. $x\in I(k_0)$), contradicting the fact that $x\in S(k_0)$. We deduce that $\langle p, \nu (x)\rangle = 0$ for every $p\in D^*d(x)$. 
This completes the proof of Lemma \ref{l:np}. 
\end{proof}
\begin{remark}\label{r:dsns}
We point out that, in the proof of \eqref{eq:diseq}, we have made no use of the convexity of $\mathcal{O}$. 
\end{remark}

We now continue  the proof of Theorem~\ref{t2}.
Recalling that  $\mathcal{O}$ is convex and $x_0\in \partial\mathcal{O}\cap S(k_0)$, by Lemma \ref{l:np} and \eqref{eq:sc} we deduce that 
\begin{equation}\label{eq:gp}
\forall p\in D^+E(d)(x_0):\, \langle \nu (x_0),p\rangle = 0. 
\end{equation}
%
 Let $p_0$ be the unique point in $D^+E(d)(x_0)$ such that 
$$
\min_{ p\in D^+E(d)(x_0)}|p|^2=|p_0|^2.
$$
Observe that $p_0\not\in D^*E(d)(x_0)$ and $-\nu (x_0)$ belongs to the normal cone to $D^+E(d)(x_0)$ at $p_0$, $N_{D^+E(d)(x_0)}(p_0)$.  
Then, by \cite[Theorem~1.5 ]{ABC} (see also \cite[Theorem~4.2 ]{A1}), we deduce that there exist a positive number $\sigma$ and a map, $[0,\sigma [\ni s\to x(s)\in \Sigma (d)$, continuous at $s=0$, such that $x(0)=x_0$,  $x(s)\not=x_0$, for every $s\in [0,\sigma [$, and 
\begin{equation}\label{eq:dirproco}
x(s)=x_0+s\nu (x_0) +o(s),\qquad \text{ as }s\to 0^+.
\end{equation}
This completes the first step of the proof. 

\smallskip
Let $s_j\in ]0,\sigma [$ be a sequence converging to $0$ as $j\to\infty$. 
Then, we may apply Theorem~\ref{t3} (ii) to each $x(s_j)\in \Sigma (d)\setminus \mathcal{O}$. Then, we find a sequence of Lipschitz arcs $x_j:[0,\infty [\longrightarrow \Sigma (d)$, with velocities bounded by 1,  such that $x_j(0)=x(s_j)$ and $d(x_j(t_1))<d(x_j(t_2))$,  for every $0\le t_1<t_2$. 
Therefore, for every $T>0$ there exists $R>0$ such that, for every $j\in \N$,
\begin{enumerate} 
\item $x_j(t)\in B_R(x_0)$, for every $t \in [0,T]$,
\item $|x_j(t)-x_j(s)|\le |t-s|$, for every $t,s\in [0,T]$.
\end{enumerate}
Then, in view of the compactness of trajectories to a differential inclusion (see, e.g., \cite[Theorem~3.1.7]{C}), we  deduce that there exists $x(t):=\lim_{j\to \infty }x_j(t)$, uniformly on  $[0,T]$, with
$$
x(0)=x_0\quad\mbox{and}\quad \dot{x}(t)\in D^+d(x(t)),\quad \text{ for a.e. }t\in [ 0,T]. 
$$
By construction, $x(t)\in \overline{\Sigma(d)}$ for every $t\in [0,T]$. 
Let us assume by contradiction that $x(t_0)\notin \Sigma (d)$, for a suitable $t_0\in [0,T]$. Then, by \cite{A5}, we deduce that $x(t)\notin \Sigma (d)$ and 
$\dot x(t_0-t)=-Dd\big(x(t_0-t)\big)$ for all $t\in ]0,t_0[$. Consequently,  for all $t\in ]0,t_0[$
\begin{equation}
\label{eq:length-min}
d\big(x(t_0-t)\big)-d\big(x(t_0)\big)=\int_0^t\frac{d}{ds}d\big(x(t_0-s)\big)ds=-t.
\end{equation}
Let $\gamma_0:[0,d(x_0)]\to \R^n$ be a length minimizer with $\gamma_0(0)=x_0$. By \eqref{eq:length-min}
$$
\gamma (t):=
\begin{cases}
x(t_0-t),\qquad &\text{ for }t\in [0,t_0],
\\
\gamma_0(t-t_0),\qquad &\text{ for }t\in ]t_0, t_0+d(x_0)].
\end{cases}
$$
 is in turn a length minimizer  on $[0, t_0+d(x_0)]$. On the other hand,  $\gamma (t_0)\in \Sigma (d)$,  in contrast with the differentiability of $d$ along length minimizers (see \eqref{eq:rlt} in Lemma \ref{l:vi}). Therefore,  $x(t)\in \Sigma (d)$ for every $t\in [0,T]$.  
Since $T$ is an arbitrary positive number, we conclude that $x(t)\in \Sigma (d)$, for every $t\geq 0$. 
The assertions concerning the injectivity of the map $[0,+\infty [\ni t\mapsto x(t)$ and the unboundedness of $\{ x(t)\ | \ t\geq 0\}$ follow arguing as in the proof of Theorem \ref{t3} (ii).
%
 
 \subsection{Proof of Theorem \ref{t2nc}} 
 Let $x_0\in \Sigma (d)\cap \partial\mathcal{O}$. We want to show that there exists a nonconstant singular arc, starting at  $x_0$, and lying in $\R^n\setminus \mathcal O$ (except for the initial point). First, suppose   $\langle p,\nu(x_0)\rangle =0$, for every $p\in D^*d(x_0)$. Then, arguing as in the  proof of Theorem \ref{t2} (see the reasoning leading to \eqref{eq:dirproco}),   
we deduce that there exist a positive number $\sigma$ and a map, $[0,\sigma [\ni s\to x(s)\in \Sigma (d)$, such that 
\begin{itemize} 
\item[$(i)$] $x(\cdot )$ is continuous at $0$ and $x(0)=x_0$;
\item[$(ii)$]  $x(s)\not=x_0$, for every $s\in [0,\sigma [$;
\item[$(iii)$] $x(s)=x_0+s\nu (x_0) +o(s),\qquad \text{ as }s\to 0^+$.
\end{itemize} 
Notice that, Condition (iii) above ensures that, possibly reducing $\sigma$ and for every $s\in ]0,\sigma ]$, $x(s)\in \R^n\setminus \mathcal O$.  

So, let $\langle p_1, \nu (x_0)\rangle <0$, for some $p_1\in D^*d(x_0)$\footnote{Recall that $\langle p, \nu (x_0)\rangle \leq 0$ for all $p\in D^*d(x_0)$ (see Remark~\ref{r:dsns}).}. 
 Then, in order to apply   Theorem \ref{t:fleming} to $d$ at $x_0$, 
 it remains to show that 
 there exists $p_0\in \co D^*d(x_0) \setminus D^*d(x_0)$ such that 
 \begin{equation}\label{eq:dirprow} 
 p_0-t\nu (x_0)\notin \co D^*d(x_0),\qquad \forall t>0. 
 \end{equation} 
 We argue by contradiction assuming that 
 \begin{equation}\label{eq:nondirprow}
 \forall  p\in \co D^*d(x_0) \setminus D^*d(x_0)\quad \exists t>0:\quad p-t\nu (x_0)\in \co D^* d(x_0).
 \end{equation} 
 Since $ \co D^*d(x_0) $ is a compact set, from \eqref{eq:nondirprow} it follows that
 \begin{equation}\label{eq:la}
\forall  p\in \co D^*d(x_0) \setminus D^*d(x_0)\quad \exists t>0:\quad p-t\nu (x_0)\in D^* d(x_0) 
 \end{equation} 
for otherwise, by taking $t$ large enough, one would intersect the boundary of $ \co D^*d(x_0) \setminus D^*d(x_0)$ at a point at which \eqref{eq:nondirprow} fails. 

Notice that, since $x_0\in \Sigma (d)$,  there exists $p_2\in D^*d(x_0)\setminus \{ p_1\}$.  Hence,  by \eqref{eq:la}, we find that 
 for every $\lambda \in ]0,1[$ there exists $t_\lambda >0$ such that 
 \begin{equation}\label{eq:p1p2} 
p(\lambda ):= p_1 +\lambda (p_2-p_1) -t_\lambda \nu (x_0)\in D^*d(x_0). 
 \end{equation} 
 Moreover, recalling that $|p(\lambda )|^2=1$ if $p(\lambda )\in D^*d(x_0)$, we have that 
 $$
t_\lambda^2 - 2t_\lambda \langle \nu (x_0), p_1+\lambda (p_2-p_1)\rangle + |p_1+\lambda (p_2-p_1)|^2-1 =0.
 $$
Hence,
 \begin{multline}\label{eq:tlambda}
 t_\lambda =\langle \nu (x_0),p_1+\lambda (p_2-p_1)\rangle
 \\
 +\sqrt{  ( \langle p_1+\lambda (p_2-p_1),\nu (x_0)\rangle)^2 +1-|p_1 +\lambda (p_2-p_1)|^2 },  
 \end{multline}
 for every $\lambda\in ]0,1[$. 
Then, we get 
 \begin{multline*}
 \langle \nu (x_0) , p(\lambda )\rangle =\langle \nu (x_0) ,   p_1 +\lambda (p_2-p_1) \rangle -t_\lambda 
\\
=-\sqrt{  (\langle p_1 , \nu (x_0) \rangle +\lambda \langle  p_2-p_1 ,\nu (x_0)\rangle )^2 +1-| p_1+\lambda (p_2-p_1)|^2 }.  
\end{multline*}
Notice that  $(\langle  p_2-p_1 ,\nu (x_0)\rangle )^2-|p_2-p_1|^2 <0$. Indeed, we clearly  have that
$(\langle  p_2-p_1 ,\nu (x_0)\rangle )^2-|p_2-p_1|^2 \leq 0$. Should equality hold, one would have that 
$p_2=p_1+c\nu(x_0)$ for some $c\neq 0$. Then, since $|p_2|=1=|p_1|$, one would deduce that $c=-2\langle p_1,\nu(x_0)\rangle $.
This would in turn yield $\langle p_2,\nu(x_0)\rangle =-\langle p_1,\nu(x_0)\rangle >0$, while we know that $\langle p_2,\nu(x_0)\rangle \leq 0$.

Therefore, the function under the square root above is quadratic w.r.t. the variable $\lambda$.        
Thus, it is strictly monotone on a suitable connected open interval, $I\subset ]0,1[$. Furthermore, we deduce that also the continuous function,  
$$
I\ni \lambda \mapsto \langle \nu (x_0) , p(\lambda )\rangle ,\text{ is strictly monotone.} 
$$ 
 In other words, we have found that if \eqref{eq:la} holds then  
 \begin{equation}\label{eq:abs} 
 J:=\{ \langle \nu (x_0) , p(\lambda )\rangle \ | \ \lambda \in I\}\subset \R\text{ has positive measure.} 
 \end{equation} 
 So, the proof reduces to show that \eqref{eq:abs} fails. For this purpose, set 
 $$
\mathcal{S}^{n-1}_{-}=\{ p\in \R^n\ | \ |p|=1\text{ and }\langle p, \nu (x_0)\rangle <0\}
 $$
 and observe that  $D^*d(x_0)\subseteq \overline{\mathcal{S}}^{n-1}_{-}$.
   For $\delta>0$, we define 
 $$
 S_\delta :=\{ p\in  \mathcal{S}^{n-1}_{-}\ |\    \langle p,\nu (x_0)\rangle <-\delta \}, 
 $$
and fix  $\delta \in ]0,1[$ such that 
$$
p(\lambda ) \in  S_\delta ,\qquad \forall \lambda \in I,   
$$
where $p(\lambda )\in D^*d(x_0)$ are the points given in \eqref{eq:p1p2}.
 Since $\partial \mathcal{O}\in C^2$, there exist $r>0$  such that the ball $B_r( x_0+r\nu(x_0))$ is tangent (from the exterior) to $\partial \mathcal O$ at $x_0$.  Observe that, if $p\in S_\delta$ and $x_0-tp\in \partial \mathcal O$, for some $t>0$, then we have that $x_0-tp\notin   B_r( x_0+r\nu(x_0))$. 
This last condition can be written as $| tp +r\nu (x_0)|^2>r^2$, i.e.,  $t> -2r\langle p,\nu (x_0)\rangle$. 
 Hence, we find that if $p\in S_\delta$ and $x_0-tp\in \partial \mathcal O$ for some $t>0$, then  
 $$
 t>2r\delta =: r_0. 
 $$

Let us define  
\begin{equation}
\label{eq:X}
X=\Big\{ x\in \partial \mathcal{O}\ | \    u(x):=\frac{x_0-x}{|x-x_0|}\in S_\delta \Big\}. 
\end{equation}
We point out that, as a consequence of the above considerations,  $X$ is a nonempty open subset of $\partial\mathcal O\setminus \overline{B}_{r_0}(x_0)$. Furthermore,  $u:X\longrightarrow S_\delta$ is a smooth map between manifolds of the same dimension, $n-1$. 
We shall now use a classical argument of Differential Topology (see, e.g., \cite[page 88, Exercise 7]{GP}). 
 For $p\in S_\delta$, let $r(p)=\{ x_0-tp\ | \ t\geq 0\}$.   
Observe that the ray $r (p)$ is transversal to $X$  if and only if 
$p$ is a regular value\footnote{We recall that  a value $p$ is regular for $u$ if for every $x\in X$ such that $u(x)=p$ the differential $du(x)$ is surjective between the tangent space to $X$ at $x$ and the tangent space to $S_\delta$ at $u(x)$. Furthermore, since these tangent spaces have the same dimension, $n-1$,  $du(x)$ is surjective if and only if it is injective.} for the smooth map $u:X\to S_\delta$ defined  in \eqref{eq:X}.
%
%
In order to show that this is indeed the case, we note that, for every $x\in X$ and  every tangent vector, $\xi $, to $X$ at $x$, 
\begin{eqnarray*}
du(x)(\xi )&=&-\frac 1{|x-x_0|} \left (\xi -  \langle x-x_0, \xi \rangle \frac{x-x_0}{|x-x_0|^2}\right )
\\
&=&-\frac 1{|x-x_0|}\,  \Big (\xi -  \langle u(x), \xi \rangle  \, u(x)\Big ).
\end{eqnarray*}
So, $u(x)=:p$ is a critical value for $u$ (i.e., $du(x)$ fails to be injective) if and only if there is a tangent vector to $X$ at $x$, say $\xi$, such that 
$$
\xi - \langle p,\xi \rangle  p=0\,.
$$
Therefore, $p$ is a tangent vector to $X$ at $x$ as well.  
In particular, the set of the critical values of $u$ contains $\{ p(\lambda )\ | \ \lambda \in I\}$. Now, let us define the smooth function 
$$
f:X\longrightarrow \R,\quad f(x)= \langle \nu (x_0), u(x)\rangle , \quad (x\in X),   
$$
and observe that if a point is critical for $u$ then it is so for $f$. 
Then, $J$ is a subset of the critical values of $f$. Thus, by Sard's Lemma, we deduce the contradiction that  $J\subset \R$ is a zero measure set but, due to  \eqref{eq:abs}, it is of positive measure too.  
So, \eqref{eq:dirprow}  follows and the proof of Theorem \ref{t2nc} is completed.

\subsection{Proof of Theorem \ref{t:exsing}}
We want to show that $\partial \mathcal{O}\cap \Sigma (d)\not=\emptyset$. 
First of all, we point out that, by \eqref{eq:dini},  $d$ is differentiable on $I(k_0)\setminus\{k_0\}$. Therefore  
\begin{equation}\label{eq:nsi}
 \Sigma (d)\setminus\{k_0\}\subset  S(k_0).
\end{equation}

 In order to find a point in $ \Sigma (d)\cap\partial \mathcal{O}$, we take a constrained maximum point for $d$ on $\partial \mathcal{O}$. Let $x_0\in \partial \mathcal{O}$ be such a point.  
 We claim that $x_0\in S(k_0)$. In order to prove our claim, we argue by contradiction assuming that $x_0\in I(k_0)$. Then, 
$$
r_0:=d(x_0)=|k_0-x_0|.
$$ 
Since 
$$
|x-k_0|\le d(x)\le d(x_0),\qquad \forall x\in \partial \mathcal{O}, 
$$
we deduce that 
\begin{equation}\label{ain}
\partial \mathcal{O}\subset \overline{B_{r_0}(k_0)}\quad \text{ and, by convexity, }\quad \mathcal{O}\subset \overline{B_{r_0}(k_0)}.
\end{equation}
Since $\partial \mathcal{O}$ is of class $C^2$,  there exist $\delta>0$ and $\phi$ of class $C^{2}$ such that 
$$
\partial \mathcal{O}\cap B_\delta (x_0)=\{ x\in B_\delta (x_0)\ | \ \phi (x)=0\text{ and } D\phi (x)\not= 0\}  
$$
and 
$$
 \mathcal{O}\cap B_\delta (x_0)=\{ x\in B_\delta (x_0)\ | \ \phi (x)<0\}.  
$$
Since $x_0\in \partial \mathcal{O}\cap \partial B_{r_0} (k_0)$,   we have that %
$$
\frac{ x_0-k_0}{|x_0-k_0|}=\frac{D\phi (x_0)}{|D\phi (x_0)|}
$$
because, by \eqref{ain}$, \partial \mathcal{O}$ and $ \partial B_{r_0} (k_0)$ have the same tangent plane at $x_0$. 
Then, we find that, for $\epsilon >0$ suitably small, 
\begin{multline*}
x_0-\epsilon \frac{D\phi (x_0)}{|D\phi (x_0)|}=x_0-\epsilon \frac{(x_0-k_0)}{|x_0-k_0|}
\\
=\left ( 1-   \frac{ \epsilon}{|x_0-k_0|} \right )    x_0+ \frac{ \epsilon}{|x_0-k_0|} k_0\in \text{ int } \mathcal{O},  
\end{multline*}
where $\text{ int } \mathcal{O}$ denotes the interior of $\mathcal{O}$. 
Hence, we find the contradiction  $\text{ int } \mathcal{O}\cap [k_0,x_0[\not=\emptyset$ and $x_0\in I(k_0)$. Our claim $x_0\in S(k_0)$ follows. 

Next, we proceed to show that $x_0\in \Sigma (d)$. 
Assume, by contradiction, that $x_0\notin \Sigma (d)$.  
Then,  by the Lagrange multiplier rule in Lipschitz settings (see, e.g., \cite[Theorem 6.1.1]{C}), we find that the unit vector $Dd(x_0)$ is parallel to the outward unit normal to the obstacle $\mathcal{O}$ at $x_0$,  $\nu (x_0)$. Moreover, Lemma \ref{l:vi} ensures that $Dd(x_0)=-\nu (x_0)$ and there exists a length minimizer of the form $\gamma (t)=x_0+t\nu (x_0)$ for $t\in [0,t_*]$, with either $t_*=d(x_0)$ or $t_*$ is the first positive time $t$ such that $\gamma (t)\in \mathcal{O}$.  Now, both alternatives are impossible: the first one because it would imply that $x_0\in I(k_0)$, the second one because $\mathcal{O}$ is convex. So,  $x_0\in \Sigma (d)$. 
This completes our proof of Theorem \ref{t:exsing}.

 \appendix 
  
  \section{Proof of Theorem \ref{t:rb}}

Let $K\subset \R^n$ be a compact set such that $k_0\notin K$ and $K\cap \partial \mathcal{O}\not=\emptyset$. For $x\in K\cap \overline{\R^n\setminus \mathcal{O}}$ we define 
$$
E(x)=\inf_\gamma \int_0^1 \langle A(\gamma (t))\dot{\gamma}(t),\dot{\gamma}(t)\rangle  \, dt
$$
where the infimum is taken w.r.t. $\gamma \in AC([0,1];\overline{\R^n\setminus \mathcal{O}})$ with $\gamma (0)=x$ and $\gamma (1)=k_0$. 
We recall that
\begin{equation}\label{eq:ed}
d(x)=\sqrt{E(x)},\quad x\in K\cap \overline{\R^n\setminus \mathcal{O}} .
\end{equation}
Equation \eqref{eq:ed} describes a well-known property whose proof can be found, for instance, in \cite[page 93]{GHL}.  
Now, observe that, since $d>0$ on $K\cap \overline{\R^n\setminus \mathcal{O}}$, we have that  $E>0$ on $K\cap \overline{\R^n\setminus \mathcal{O}}$.
Furthermore, in \cite[Proposition 3.9]{CCMW}, it is shown that $E\in SC^{\frac 12}(K\cap \overline{\R^n\setminus \mathcal{O}})$. 
Then, it suffices to prove  that the square root of a positive semiconcave function is semiconcave. 
For this purpose, first we show that $\sqrt{E}$ is a Lipschitz function on $K\cap \overline{\R^n\setminus \mathcal{O}}$. 
Indeed, the existence of a constant $L$ such that 
$$
|E(x)-E(y)|\leq L |x-y|,\qquad \forall x,y\in K\cap \overline{\R^n\setminus \mathcal{O}}
$$
implies that, for every $x,y\in K\cap \overline{\R^n\setminus \mathcal{O}}$, 
$$
|\sqrt{E(x)}-\sqrt{E(y)}|\leq \frac {L}{\sqrt{E(x)}+\sqrt{E(y)}} |x-y|\leq L'|x-y| ,  
$$
 with 
 $$
 L'=\frac{L}{2\min_{x\in K\cap \overline{\R^n\setminus \mathcal{O}}} \sqrt{E(x)}}. 
 $$
 Furthermore, the fact that $E\in SC^{\frac 12}(K\cap \overline{\R^n\setminus \mathcal{O}})$ implies that, for every $x\in K\cap \overline{\R^n\setminus \mathcal{O}}$, we have that 
 $$
E(y)-E(x)\le \langle p, y-x\rangle +C |y-x|^{3/2}
 $$
 for every $y\in K\cap \overline{\R^n\setminus \mathcal{O}}$ such that $[x,y]\subset K\cap \overline{\R^n\setminus \mathcal{O}}$ and for every $p\in D^+E(x)$. Then, 
 \begin{multline}\label{eq:sca2}
 \sqrt{E(y)}\le \sqrt{E(x)} +\frac 1{\sqrt{E(x)}+\sqrt{E(y)}} (\langle p, y-x\rangle +C |y-x|^{3/2})
 \\
\le \sqrt{E(x)} +\frac 1{2\sqrt{E(x)}} \langle p, y-x\rangle +\left (  \frac 1{\sqrt{E(x)}+\sqrt{E(y)}}  -   \frac 1{2\sqrt{E(x)}}\right ) \langle p, y-x\rangle  
\\
+\frac {C}{\sqrt{E(x)}+\sqrt{E(y)}} |y-x|^{3/2}
\\
\le \sqrt{E(x)} +\frac 1{2\sqrt{E(x)}} \langle p, y-x\rangle +
C' |y-x|^{3/2} 
 \end{multline}
 with 
 $$
C'=\max_{x,y\in K\cap \overline{\R^n\setminus \mathcal{O}}}  \left ( \frac{C}{\sqrt{E(x)}+\sqrt{E(y)}}    +    \frac{ L L' \sqrt{|y-x|} }{2\sqrt{E(x)} [   \sqrt{E(x)}+\sqrt{E(y)}]} \right ),  
 $$
 i.e. for every $x,y\in K\cap \overline{\R^n\setminus \mathcal{O}}$ such that $[x,y]\subset K\cap \overline{\R^n\setminus \mathcal{O}}$ and for every $p\in D^+E(x)$
 \begin{equation}\label{eq:estfin}
 \sqrt{E(y)}\le \sqrt{E(x)} +\frac 1{2\sqrt{E(x)}} \langle p, y-x\rangle +
C' |y-x|^{3/2}. 
 \end{equation}
Now, let $x,y\in K\cap \overline{\R^n\setminus \mathcal{O}}$ such that $[x,y]\subset K\cap \overline{\R^n\setminus \mathcal{O}}$ and let $\lambda\in [0,1]$. Then, by \eqref{eq:estfin}, with $x$ replaced with $\lambda x+(1-\lambda )y$, we find 
 \begin{multline}\label{0}
\sqrt{E(y)}\le \sqrt{E(\lambda x+(1-\lambda )y )} 
 \\
 +\frac 1{2\sqrt{E(\lambda x+(1-\lambda )y)}}\lambda  \langle p, x-y\rangle +
C'\lambda^{3/2} |y-x|^{3/2},
 \end{multline}
 with $p\in D^+d( \lambda x+(1-\lambda )y)$. Analogoulsy, we find that 
 \begin{multline}\label{1}
\sqrt{E(x)}\le \sqrt{E(\lambda x+(1-\lambda )y )} 
 \\
 +\frac 1{2\sqrt{E(\lambda x+(1-\lambda )y)}}(1-\lambda ) \langle p, y-x\rangle +
C'(1-\lambda)^{3/2} |y-x|^{3/2}. 
 \end{multline}
 Taking the convex combination of \eqref{0} with \eqref{1}, we find that 
 \begin{multline*}\label{01}
\lambda \sqrt{E(x)}+(1-\lambda ) \sqrt{E(y)}- \sqrt{E(\lambda x+(1-\lambda )y )} 
 \\
\le C'   ((1-\lambda ) \lambda^{3/2}+\lambda (1-\lambda)^{3/2}) |y-x|^{3/2}\le 
 C'   (1-\lambda ) \lambda  |y-x|^{3/2}. 
 \end{multline*} 
 This completes our proof.

\begin{remark}
The above proof is based on the property that the square root of the positive function $E\in SC^{\frac 12}(K\cap \overline{\R^n\setminus \mathcal{O}})$ is itself fractionally semiconcave of order $1/2$. We note that, here, known results dealing with the semiconcavity of a composition (see, e.g., \cite[Proposition~2.1.12 (i)]{CS}) cannot be applied  for two reasons. First, they do not address semiconcave functions on a closed domain and, second, they do not provide a control of the fractional semiconcavity modulus.

\end{remark}

\section*{Declarations}
\begin{itemize}
\item {\bf Conflict of interest:} On behalf of all authors, the corresponding author states that there is no conflict of interest. 
\item{\bf Funding:} This work was partly supported by the National Group for Mathematical Analysis, Probability and Applications (GNAMPA) of the Italian Istituto Nazionale di Alta Matematica ``Francesco Severi''; moreover, the third author acknowledges support by the Excellence Department Project awarded to the Department of Mathematics, University of Rome Tor Vergata, CUP E83C18000100006. 
\item{ \bf Acknowledgment:} The authors are very grateful to the anonymous referee who provided constructive criticism that highly improved the quality of the paper. 
\end{itemize}

\end{document}